\documentclass[reqno]{amsart}
\usepackage{enumerate}
\usepackage{scrextend}
\usepackage{amsmath,amsxtra,amssymb,latexsym, amscd,amsthm,pb-diagram,fancyhdr,euscript}
\usepackage{amsthm}
\usepackage{latexsym}
\usepackage{amssymb}
\usepackage{amsmath}
\usepackage{indentfirst}
\usepackage{hyperref}
\usepackage{mathrsfs}
\usepackage{array}
\usepackage{euscript}
\usepackage{slashed}
\usepackage[T1]{fontenc}
\usepackage{graphicx}
\usepackage[utf8]{inputenc}
\usepackage[french,english]{babel}
\usepackage{lipsum} 
\usepackage{multicol}
\hypersetup{
   colorlinks,
    citecolor=blue,
    filecolor=black,
    linkcolor=blue,
    urlcolor=magenta
}
\hypersetup{linktocpage}
\changefontsizes[16pt]{12pt}
\newtheorem{theorem}{Theorem}[section]
\newtheorem{lemma}[theorem]{Lemma}

\theoremstyle{definition}
\newtheorem{definition}[theorem]{Definition}

\newtheorem{remark}[theorem]{Remark}

\newtheorem{assumption}[theorem]{Assumption}

\newcommand{\norm}[1]{\left\Vert#1\right\Vert}

\numberwithin{equation}{section}
\usepackage{amsfonts}

\usepackage{amsmath}
\usepackage{amssymb}

\begin{document}
\font\nho=cmr10
\def\dive{\operatorname{div}}
\def\cal{\mathcal}
\def\L{\cal L}
\def\M{\cal M}
\def\grad{\operatorname{grad}}
\def \ud{\underline }
\def\id{{\indent }}
\def\f{\frac}
\def\non{{\noindent}}
\def\le{\leqslant} 
\def\leq{\leqslant}
\def\geq{\geqslant} 
\def\rar{\rightarrow}
\def\Rar{\Rightarrow}
\def\ti{\times}
\def\i{\mathbb I}
\def\j{\mathbb J}
\def\si{\sigma}
\def\Ga{\Gamma}
\def\ga{\gamma}
\def\ld{{\lambda}}
\def\Si{\Psi}
\def\f{\mathbf F}
\def\r{\hro{R}}
\def\e{\cal{E}}
\def\B{\cal B}
\def\A{\cal A}
\def\bb{\mathbb B}
\def\p{\mathbb P}
\def\RR{\mathcal{R}}
\def\tet{\theta}
\def\Tet{\Theta}
\def\hro{\mathbb}
\def\ho{\mathcal}
\def\P{\ho P}
\def\E{\mathcal{E}}
\def\n{\mathbb{N}}
\def\dMu{\mathbf{U}}
\def\dMcs{\mathbf{C}}
\def\dMcu{\mathbf{C^u}}
\def\vk{\vskip 0.2cm}
\def\td{\Leftrightarrow}
\def\df{\frac}
\def\Wei{\mathrm{We}}
\def\Rey{\mathrm{Re}}
\def\s{\mathbb S}
\def\l{\mathcal{L}}
\def\C+{C_+([t_0,\infty))}
\def\o{\cal O}
%\def\@frontmatterwidth{16cm}
%\textheight=22cm
%\oddsidemargin=0cm
%\evensidemargin=0cm
%\begin{frontmatter}
\title[AAP-mild solutions for NSE on non-compact manifolds]{On asymptotically almost periodic  mild solutions for Navier-Stokes equations on non-compact Riemannian manifolds}

\author[N.T. Van]{Nguyen Thi Van}
\address{Thi Van Nguyen\hfill\break
Faculty of Computer Science and Engineering, Department of Mathematics, Thuyloi University, 175 Tay Son, Dong Da, Ha Noi, Viet Nam}
\email{van@tlu.edu.vn}

%\author[T.V. Thuy]{ Tran Van Thuy}
%\address{ Tran Van Thuy\hfill\break
%East asia university of Technology,
%Trinh Van Bo street, Nam Tu Liem, Ha Noi, Viet Nam}
%\email{thuytv@eaut.edu.vn}

\author[P.T. Xuan]{Pham Truong Xuan}
\address{Pham Truong Xuan \hfill\break
Thang Long Institute of Mathematics and Applied Sciences (TIMAS), Thang Long University, Nghiem Xuan Yem, Hanoi, Vietnam 
} 
\email{xuanpt@thanglong.edu.vn or phamtruongxuan.k5@gmail.com}

%\author[N.T. Loan]{Nguyen Thi Loan}
%\address{Nguyen Thi Loan\hfill\break Faculty of Fundamental Sciences, Phenikaa University, Hanoi 12116 Vietnam}
%\email{loan.nguyenthi2@phenikaa-uni.edu.vn}

%\author[L.T. Sac]{Le The Sac}
%\address{Le The Sac\hfill\break
%Faculty of Information Technology, Department of Mathematics, Thuyloi University\\
%Khoa Cong nghe Thong tin, Bo mon Toan, Dai hoc Thuy loi, 175 Tay Son, Dong Da, Ha Noi, Viet Nam}
%\email{SacLT@tlu.edu.vn}
\begin{abstract}  
In this paper, we study the existence, uniqueness and asymptotic behaviour of almost periodic and asymptotically almost periodic mild solutions to the incompressible Navier-Stokes equations on $d$-dimensional non-compact manifold $(\mathcal{M},g)$ which satisfies some bounded conditions on curvature tensors. First, we use the $L^p-L^q$-dipsersive and smoothing estimates of the Stokes semigroup to prove Massera-type principles which guarantees the well-posedness of almost periodic and asymptotically almost periodic mild solutions for the inhomogeneous Stokes equations. Then, by using fixed point arguments and Gronwall's inequality we establish the well-posedness and exponential decay for global-in-time of such solutions of Navier-Stokes equations. Our results extend the previous ones \cite{Xuan2022,Xuan2023} to the generalized non-compact Riemannian manifolds.
\end{abstract}



\maketitle

\tableofcontents

\section{Introduction}
In the present paper we will study the existence and asymptotic behaviour of almost periodic (AP-) and asymptotically almost periodic (AAP-) mild solutions to incompressible Navier-Stokes equations (NSE) for vector fields on the $d$-dimensional non-compact Riemannian manifold $(\M,g)$ with negative Ricci curvature tensor which satisfies Assumption \ref{assum} below (see Subsection \ref{NSERiemann}).
In particular, denoting $\Gamma(T\M)$ the set of all vector fields on $\M$, we consider NSE for a vector-field-valued mapping $u:\r_+\times \mathcal{M} \to \Gamma(T\mathcal M)$
  \begin{align}\label{NSE1}
\begin{cases}
\partial_t u + \nabla_u u + \nabla \pi = \overrightarrow{\Delta}u + r(u) +  f,\cr
\mathrm{div}u = 0,\cr
u(0,x) = u_0(x) \in \Gamma(T\M) \hbox{ for all } x\in \mathcal{M}
\end{cases}
\end{align}
where $\nabla$ denotes the Levi-Civita connection on $\M$;
 $\pi$ is the pressure; $r(u)$ is Ricci operator (see the definition in the next section); 
 $f$ is external force; and $\overrightarrow{\Delta}$ is Bochner Laplace operator.

First, we recall some results about Navier-Stokes equations on Riemannian manifolds. 
In particular, the non-uniqueness problem of Leray-weak solutions of Navier-Stokes equations was proved by Czubak et al. \cite{Cz1,Cz2,Cz3} on two-dimensional hyperbolic space $\mathbb{H}^2$. Then, this non-uniqueness problem was revisited by Khesin and Misiolek \cite{Khesin2012} in $\mathbb{H}^d$ for $d\geqslant 2$ by using harmonic forms. In particular, they proved that the Navier-Stokes equations have non-uniqueness Leray-weak solutions in $\mathbb{H}^2$ and have uniqueness Leray-weak solutions in $\mathbb{H}^d$ for $d\geqslant 3$. The non-uniqueness problem were extended to study on another three-dimensional non-compact Riemannian manifold by Lichtenfelz \cite{Li2016} by considering Navier-Stokes equations on Anderson's manifold. 
After that, Pierfelice \cite{Pi} provided $L^p-L^q$- dispersive and smoothing estimates for Stokes semigroups on real hyperbolic and generalized non-compact manifolds with negative Ricci curvatures. Then, the author used these estimates and Kato-iteration method (see \cite{Ka}) to prove the existence and uniqueness of strong mild solutions for Navier-Stokes equations (see Section 5 in \cite{Pi}). There are some other works for Navier-Stokes on Riemannian manifolds such as \cite{Fa2018,Fa2020,MiTa2001,Sa,Zha} and references therein.

Concerning the existence and asymptotic behaviour of periodic mild solutions, the existence and stability of periodic solutions to Navier-Stokes equations \eqref{NSE1} on Einstein manifolds with negative curvature tensor were established in \cite{HuyXuan2018}. These results have extended to the case of general non-compact Riemannian manifolds satisfying Assumption \ref{assum} in \cite{HuyXuan2022}. For almost periodic and forward asymptotically almost periodic mild solutions, the recent works \cite{Xuan2022,Xuan2023} provides the fully results on the existence and exponential stability (and exponential decay) for equation \eqref{NSE1} in the spaces $L^p(\Gamma(T\mathbb{H}^d))$ for all $p>1$, where $\mathbb{H}^d$ is the $d$-dimensional real hyperbolic space.

We would like to notice that, the questions on existence and asymptotic behaviours of forward asymptotically almost periodic solutions of Navier-Stokes equations on flat spaces such as the whole Euclidean space $\mathbb{R}^n$ or its unbounded domains remain open sofar. In related work, there is only work done by Farwig and Tanuichi (see \cite{FaTa}) which gives the global uniqueness of the small backward AAP- mild solution in unbounded domains in $\mathbb{R}^n$. The main obstacle is that the heat and Stokes semigroups are polynomial stable on $\mathbb{R}^n$, then the solution operator (see Definition in the proof of Theorem \ref{th1}) can not preserve the asymptotically almost periodic properties of input functions. However, the vectorial heat and Stokes semigroups are exponential stable on the framework of non-compact manifolds with negative Ricci curvatures (see \cite{Pi}). Using this fact, we can prove that the solution operator preserves the asymptotically almost periodic property (see \cite{Xuan2022,Xuan2023} for the case of hyperbolic space).

The purpose of this paper is to extend the results in \cite{Xuan2022} for some non-compact Riemannian manifolds satisfying Assumption \ref{assum} below which require some bounded conditions on curvatures (see the end of Subsection \ref{NSERiemann}). In this geometric framewrork, we will face the obstacles that the Kodaira-Hodge projection $\mathbb{P}$ does not commute with Stokes semigroup $e^{-t\mathcal{A}}$ and the fact that the almost periodic mild solutions is determined on the full line time-axis (see \cite{KoNa} for the initial notion of this type of mild solution). To overcome these obstacles, we introduce and establish the boundedness of mild solutions of inhomogeneous Stokes equations in the phase spaces $\mathcal{X}$ (for AAP-mild solutions) and $\mathbb{X}$ (for AP-mild solutions) which consist $L^p(\Gamma(T\mathcal{M}))\cap L^2(\Gamma(T\mathcal{M}))$ for $p>d$ (see Subsection \ref{main} and Lemma \ref{Thm:linear} for the definitions of $\mathbb{X}$ and $\mathcal{X}$). Then, by using $L^p-L^q$ estimates for Stokes semigroup we prove the Massera-type principles which show the existence of AP- and AAP- mild solutions of inhomogeneous Stokes equations in $\mathcal{X}$ and $\mathbb{X}$, respectively (see Theorem \ref{th1}). Note that, since the curvatures satisfy the bounded and negative conidtions in Assumption \ref{assum} below, the Stokes semigroup is exponential stable and the solution operator can preserve the asymptotic property as $t$ tends to infinity (see Theorem \ref{th1} $ii)$). Combining these results with fixed point argument we establish the existence of such solutions for Navier-Stokes equations \eqref{NSE1} (see Theorem \ref{th2}). Finally, we provide the exponential decay for AP- and AAP- mild solutions by using Gronwall's inequalities (see Theorem \ref{th3}). \\
{\bf Notations.}\\
For a given Banach space $X$ we denote the following Banach spaces:
$$C_b(\r_+, X):=\{h:\r_+\to X \mid h\hbox{ is continuous and }\sup_{t\in\r_+}\|h(t)\|_X<\infty\}$$
endowed with the norm $\|h\|_{\infty, X}=\|h\|_{C_b(\r_+, X)}:=\sup\limits_{t\in\r_+}\|h(t)\|_X$. 

$$C_b(\r, X):=\{h:\r\to X \mid h\hbox{ is continuous and }\sup_{t\in\r}\|h(t)\|_X<\infty\}$$
endowed with the norm $\|h\|_{\infty, X}=\|h\|_{C_b(\r, X)}:=\sup\limits_{t\in\r}\|h(t)\|_X$.
  
\section{Setting of geometric analysis and main results} \label{S2} 
\subsection{Some basic notions on non-compact Riemannian manifolds}\label{NSERiemann}
\noindent
For $2\le d\in\mathbb{N}$, we consider a $d$-dimensional non-compact Riemannian manifold $(\M,g)$ which satisfies the below assumption (which was introduced in \cite{Druet,Pi}):
\begin{assumption}\label{assum}
Assume that $(\mathcal{M},g)$ is a  smooth, complete, non-compact, simply connected Riemannian manifold  satisfying the following properties
\begin{itemize}
\item[$(H_1)$] $|\RR| + |\nabla \RR| + |\nabla^2 \RR| \leq K \, ,$
\item[$(H_2)$] $-\dfrac{1}{c_0}g \leq \mathrm{Ric} \leq -c_0 g \, ,$ for some $c_0$ positive,
\item[$(H_3)$] $\kappa <0$,
\item[$(H_4)$] $\inf_{x\in \mathcal{M}} r_x >0$,
\end{itemize}
where $r_x$ is the injectivity radius for the exponential map at $x$.
\end{assumption} 
\noindent
{\bf Examples:} There are several Riemannian manifolds satisfying the above hypotheses $(H_1)-(H_4)$ such as 
real hyperbolic manifolds,  non-compact Einstein manifolds with negative Ricci curvature tensors  (\cite{Hel,Jo}), Damek-Ricci manifolds (\cite{Da}) and symmetric manifolds of non-compact types (\cite{Er, Hel}).

We refer the reader to \cite{Jo,He} for basic notions and operators on Riemannian manifolds. %For details on Navier-Stokes equations on Riemannian manifolds we refer to \cite{EbiMa,Pi,Tay} and references therein. 
For convenience, in what follows, we recall only some notions on differential operators on Riemannian manifolds.  We denote the Levi-Civita connection by $\nabla$ and the set of all vector fields on $\M$ by $\Gamma(T\M)$. For $X\in \Gamma(T\M)$ we can extend  $\nabla_X$ to arbitrary $(p,q)$
  tensor by requiring
  \begin{enumerate}
  \item $\nabla_X(c(S)) = c(\nabla_X S)$ for any contraction $c$,
  \item $\nabla_X(S \otimes T) = \nabla_X S \otimes T + S \otimes \nabla_X T$
\end{enumerate}   
where we take the convention that $\nabla_X f = X\cdot f$ for a function $f:\M\to \r$. 

In particular, for 
$S\in \Gamma(\otimes^p(T\M)\otimes^q
(T^*\M))$  we get 
$$(\nabla_X S)(X_1, \cdots, X_q ) = \nabla_X(S(X_1, \cdots, X_q ))-
S(\nabla_XX_1,\cdots, X_q))-\cdots-S(X_1, \cdots,\nabla_X X_q).$$
Moreover, we define the covariant derivatives $\nabla$ on tensor field 
$S\in \Gamma(\otimes^p(T\M)\otimes^q(T^*\M))$ 
by
$$\nabla S(X, X_1, \cdots, X_q) = (\nabla _X S)(X_1, \cdots, X_q),
\hbox{ hence, }\nabla S\in \Gamma(\otimes^p(T\M)\otimes^{q+1}(T^*\M).$$
Next, we recall the ``musical isomorphisms'' on Riemannian manifolds. 
For a 1-form $w$, we define the vector field $w^\sharp$	 by
$$g(w^\sharp, Y ) = w(Y), \forall Y \in \Gamma(T\M)$$
whereas, for a vector field $X$, we define the 1-form $X^\flat$ by
$$X^\flat(Y) = g(X, Y ), \forall Y \in \Gamma(T\M).$$
The metric on 1-forms  can then be defined by setting
$$g(w, \eta) := g(w^\sharp, \eta^\sharp), \forall w, \eta \in \Gamma(T^*\M).$$
The Riemannian gradient of a function is then defined as
$$\grad p = (dp)^\sharp.$$
More generally, for $(p,q)$-tensor field $ T\in \Gamma(\otimes^p(T\M)\otimes^q(T^*\M))$
  we have
\begin{eqnarray*}
T^\sharp& = &C^2_1(g^{-1}\otimes T) \in \Gamma(\otimes^{p+1}(T\M)\otimes^{q-1}(T^*\M)),\cr
T^\flat & = &C^1_2(g \otimes  T ) \in \Gamma(\otimes^{p-1}(T\M)\otimes^{q+1}(T^*\M)),\cr
\dive T &= &C^1_1\nabla T \in \Gamma(\otimes^{p-1}(T\M)\otimes^{q}(T^*\M))
\end{eqnarray*}
where $C^i_j$ stands for the contraction of the $i$ and $j$ indices for tensors.

{ Next, the Laplace-Beltrami operator $\Delta_g$ applying on functions is defined as 
$$\Delta_g(f)=\dive\grad f =\frac{1}{\sqrt{|g|}}\frac{\partial}{\partial x^j}\left(\sqrt{|g|}g^{ij}\frac{\partial f}{\partial x^i}\right)\hbox{ for a function }f :\mathcal{M} \to \r,$$
where $|g| = \mathrm{det}\,g$.}
 
Furthermore,  the vectorial Laplacian $L$ is defined by the stress tensor (see \cite{EbiMa,Tay}) :
$$Lu = \dive(\nabla u + \nabla u^t)^{\sharp}.$$
%where $\omega^{\sharp}$ is a vector field associated with the 1-form $\omega$ by
%$$g(\omega^{\sharp},Y) = \omega(Y) \, \forall Y \in \Gamma(TM).$$
Since $\dive u=0$ we can express $L$ in the following formula
$$Lu = \overrightarrow{\Delta}u + r(u),$$
where $\overrightarrow{\Delta}$ is the Bochner-Laplacian
$$\overrightarrow{\Delta}u =- \nabla^*\nabla u= \mathrm{Tr}_g(\nabla^2u)$$
%In local coordinates one has $\mathrm{Tr}_g(\nabla^2u) = g^{ij}\nabla^2u(e_i,e_j)$ for the
 %basis
 %$\left\{e_i = \dfrac{\partial}{\partial x_i} \right\}$. 
 and $r(\cdot)$ is the Ricci operator related to the Ricci curvature tensor by the formula
$$r(u) = (\mathrm{Ric}(u,\cdot))^{\sharp} \hbox{ for all }u\in\Gamma({T\M})$$
where the Ricci curvature tensor is defined by
$$\mathrm{Ric}(X,Y)=\sum_{i=1}^dg(\mathcal{R}(X,e_i)Y,e_i)\hbox{ for all }X, Y\in\Gamma({T\M})$$
for the standard basis $\{e_i=\frac{\partial}{\partial x_i}\}_{i=1}^d$ and $\mathcal{R}$ being the curvature tensor on $M$ defined by $$\mathcal{R}(X, Y )Z :=
 -\nabla_X (\nabla_Y Z)+
\nabla_Y (\nabla_X Z)+\nabla_{[X,Y ]}Z \hbox{ for all }X, Y, Z  \in \Gamma(T\M).$$
Moreover, the sectional
curvature $\kappa$ is  defined as $\kappa(X,Y):=\frac{\RR(X,Y,X,Y)}{g(X,X)g(Y,Y)-g(X,Y)^2} $ for all $X, Y\in T_x\M$.

\subsection{Almost periodic and asymptotically almost periodic functions}
\noindent

Let $X$ be a Banach space, we recall the definitions of almost periodic and forward asymptotically almost periodic functions which have the { values} in $X$ (for details see \cite{Che}).
\begin{definition}
A function  $h \in C_b(\r, X )$ is called almost periodic function if for each $ \epsilon  > 0$, there exists $l_{\epsilon}>0 $ such that every interval of length $l_{\epsilon}$ contains at least a number $T $ with the following property
\begin{equation}
 \sup_{t \in \r } \| h(t+T)  - h(t) \|_X < \epsilon.
\end{equation}
The set of all almost periodic functions $h:\r \to X $ is denoted by $AP(\r,X)$ which is { also} a Banach space endowed with the norm 
$$\|h\|_{ AP(\r,X)}=\sup_{t\in\r}\|h(t)\|_X.$$
\end{definition}
In order to introduce the asymptotically almost periodic functions, we define the space  $C_0 (\r_+,X)$, as the set of all forward asymptotic and continuous functions $\varphi: \r_+ \to X$ { satisfying}
$$\lim_{t \to +\infty } \| \varphi(t) \|_X = 0.$$
Clearly, $C_0 (\r_+,X)$  is a Banach space endowed with the norm 
$$\|\varphi\|_{C_0 (\r_+,X)}=\sup_{t\in\r_+}\|\varphi(t)\|_X.$$
\begin{definition} 
A continuous function  $f \in C(\r_+, X )$ is said to be forward asymptotically almost periodic (in the rest we call simply by asymptotically almost periodic) if there { exists two functions}  $h \in AP(\r,X)$ and $ \varphi\in C_0(\r_+,X)$ such that
\begin{equation}
f(t) = h(t) + \varphi(t).
\end{equation}
We denote $AAP(\r_+, X)= \{f:\r_+ \to X \mid f\hbox{ is asymptotically almost periodic on $\r_+$}\}$. Observe that $AAP(\r_+,X)$ is a Banach space endowed with the norm 
$$\|f\|_{ AAP(\r_+,X)}=\|h\|_{ AP(\r, X)}+\|\varphi\|_{ C_0(\r_+,X)}.$$
\end{definition}
Note that, a periodic function or an almost periodic function is also asymptotically almost periodic, but the vice versa is not always true. For instance, let $g\in C_b(\mathbb{R},X)$, then $f(t) = \left(\sin t+ \sin \sqrt{2}t + \dfrac{1}{\sqrt t}\right)g(t)$ is an asymptotically almost periodic function in $AAP(\mathbb{R}_+,X)$ but is not periodic or almost periodic.

\subsection{Navier-Stokes equations and main results}\label{main}
\noindent

We consider the existence, uniqueness and asymptotic behaviour of almost periodic and asymptotically almost periodic mild solutions to the Cauchy problem for the incompressible Navier-Stokes equations \eqref{NSE1} on $\mathcal{M}$.
Using the Kodaira-Hodge operator $\mathbb{P} = I + \mathrm{grad}(-\Delta_g)^{-1} \mathrm{div}$ we can get rid of the pressure term $\pi$ and then obtain from \eqref{NSE1} that (see \cite{Pi,Tay}): 
\begin{align}\label{DDivNavierStokes}
\begin{cases}
\partial_t u - (\overrightarrow{\Delta}u + r(u) + G(u)) = \mathbb{P}[-\nabla_uu+f],\cr
\dive u = 0,\cr
u(0,x)= u_0(x) \in \Gamma(TM)\hbox{ for all }x\in\mathcal{M}; \dive u_0=0,
\end{cases}
\end{align}
where $G(u)=2\mathrm{grad}(-\Delta_g)^{-1}\dive(r(u))$ with $\Delta_g$ being the Laplace-Beltrami operator. The corresponding inhomogeneous Stokes equation takes the form 
\begin{align}\label{CCauchyStokes}
\begin{cases}
\partial_t u &= -\mathcal{A}u + \mathbb{P}[-\nabla_vv+f],\cr
u(0,x) &= u_0(x) \in \Gamma(T\mathcal{M}) \hbox{ for all }x\in\mathcal{M}
\end{cases}
\end{align}
for given vector-field valued mappings $v(t,\cdot)$ and $f(t,\cdot)\in  \Gamma(T\mathcal M)$,
where 
$\mathcal{A}u  = -(\overrightarrow{\Delta}u + r(u) + G(u))$, and $e^{-t\A}$ is denoted  the semigroup associated with the homogeneous Cauchy problem:  
\begin{align}\label{CauchyStokes0}
\begin{cases}
\partial_t u &= -\mathcal{A}u,\cr
u(0,x) &= u_0(x) \in \Gamma(T\mathcal{M}) \hbox{ for all }x\in\mathcal{M}, 
\end{cases}
\end{align}
\textit{i.e.}, the unique solution of the above Cauchy problem is given by $u(t)=e^{-t\A}u_0$. Here, traditionally, we denote $u(t)$ for $u(t,\cdot)$.
\begin{remark}
Note that, the operator $\overrightarrow{\Delta} + r + G$ does not commute with the Kodaira-Hodge operator
$\mathbb{P} = I + \mathrm{grad}(-\Delta_g)^{-1} \mathrm{div}$ on the generalized non-compact manifolds with all the conditions $(H_1)-(H_4)$. However, in the case of Einstein manifolds, these operator are commutated (see \cite{Pi}). 
\end{remark} 
  
To state our main results we need the following notions of mild solutions:
\begin{definition}
\noindent
\begin{itemize}
	\item[(i)]  By a \textit{mild solution on the half line time-axis} to system \eqref{CCauchyStokes}, we mean
	a mapping $u:\r_+\times \mathcal{M} \to \Gamma(T\mathcal M)$
	which satisfies  the integral equation
	\begin{equation}\label{mild:linear}
		u(t) = e^{-t\mathcal{A}}u_0 + \int_0^te^{-(t-\tau)\mathcal{A}} \mathbb{P} [-\nabla_vv+f] (\tau) d\tau \hbox{ for }t\ge 0.
	\end{equation}
 \item[(ii)]  By a \textit{mild solution on the whole line time-axis} to the first equation of system \eqref{CCauchyStokes} (we do not consider the initial data), we mean
 a mapping $u:\r \times \mathcal{M} \to \Gamma(T\mathcal M)$
 which satisfies  the integral equation
 \begin{equation}\label{mild:linear}
 	u(t) =   \int_{-\infty}^te^{-(t-\tau)\mathcal{A}} \mathbb{P} [-\nabla_vv+f] (\tau) d\tau.
 \end{equation}
\end{itemize}
\end{definition}

Since, the Stokes operator $\overrightarrow{\Delta} + r + G$ does not commute with the Kodaira-Hodge operator
$\mathbb{P} = I + \mathrm{grad}(-\Delta_g)^{-1} \mathrm{div}$ on a generalized non-compact manifold,
we establish the existence and stability of the asymptotically almost periodic mild solutions of equations \eqref{DDivNavierStokes} and \eqref{CCauchyStokes} on the following Banach space
\begin{eqnarray}\label{deftXX}
\cal X&:=&\big\{ u \in C_b(\r_+, (L^p\cap L^2)(\Gamma(T\cal M))), \nabla u \in C_b(\r_+, L^{\tilde{p}}(\Gamma(T\cal M)) \cap L^s(\Gamma(T\cal M))) \mid \cr 
&&\hbox{ The function } t\mapsto  \| u(t)\|_{L^p\cap L^2 } + [c_d(t)]^{- \left(\frac{1}{p}-\frac{1}{\tilde{p}}+\frac{1}{d} \right)}\| \nabla u(t)\|_{L^{\tilde p}} \cr
&&\hspace{4cm} +  [c_d(t)]^{- \left(\frac{1}{p}-\frac{1}{s}+\frac{1}{d} \right)}\| \nabla u(t)\|_{L^s}   \hbox{ belongs to } L^\infty({\r_+}) \big\}
\end{eqnarray}
endowed with the norm
$$ \| u \|_{\cal X} := \sup_{t\in \mathbb{R}_+} (\| u(t)\|_{L^2} + \| u(t)\|_{L^p} + [c_d(t)]^{- \left(\frac{1}{p}-\frac{1}{\tilde{p}}+\frac{1}{d} \right)}\| \nabla u(t)\|_{L^{\tilde p}} + [c_d(t)]^{- \left(\frac{1}{p}-\frac{1}{s}+\frac{1}{d} \right)} \| \nabla u(t)\|_{L^s}),$$
whereas $c_d(t): = C_0\max \left\{t^{-\frac d2},1 \right\}$.

Also, we consider the existence and stability of the  almost periodic mild solutions of equations \eqref{DDivNavierStokes} and \eqref{CCauchyStokes} without initial data condition on the following Banach space
\begin{eqnarray}\label{deftXX2}
	\mathbb X&:=&\{u \in C_b(\r, (L^p \cap L^2)  (\Gamma(T\cal M))), \nabla u \in C_b(\r, L^{\tilde{p}}(\Gamma(T\cal M)) \cap L^s(\Gamma(T\cal M)))\mid \cr
	&&\hbox{ The function } t\mapsto \| u(t)\|_{L^p\cap L^2 }  +  \tilde{\lambda}(t) \| u(t)\|_{L^{\tilde{p}}}+\hat{ \lambda}(t) \| u(t)\|_{L^{s}} \cr
	&&\hspace{9cm} \hbox{ belongs to }   L^\infty({\r}) \big\}
\end{eqnarray}
endowed with the norm
$$ \| u \|_{\mathbb X} := \sup_{t\in \mathbb{R}} \{\| u(t)\|_{L^p} +\|u(t)\|_{L^2}    + \tilde{\lambda}(t) \| u(t)\|_{L^{\tilde{p}}}+\hat{ \lambda}(t) \| u(t)\|_{L^{s}}\},$$ 
here the notation $(L^p\cap L^2) (\Gamma(T\cal M))$ stands for $L^p(\Gamma(T\cal M))\cap L^2(\Gamma(T\cal M)) $ and  $\|\cdot\|_p:= \| \cdot\|_{L^p(\Gamma(T\cal {M}))}$.
Here we assume that $d < p<s^{\prime} <\tilde{p} < s$ with $\frac{1}{2}=\frac{1}{p}+\frac{1}{\tilde{p}}$,  whereas $c_d(t): = C_0\max \left\{t^{-\frac d2},1 \right\}$ and  $$
 \tilde{\lambda}(t)= \begin{cases}
[c_d(t)]^{- \left(\frac{1}{p}-\frac{1}{\tilde{p}}+\frac{1}{d} \right)}, \text{ if } t\ge 0,  \cr
[c_d(|t|)]^{- \left(\frac{1}{p}-\frac{1}{s^{\prime}}+\frac{1}{d} \right)},\text{ if } t< 0; 
   \end{cases}
\;\hat{\lambda}(t)=\begin{cases}
	[c_d(t)]^{- \left(\frac{1}{p}-\frac{1}{s}+\frac{1}{d} \right)}, \text{ if } t\ge 0,  \cr
	[c_d(|t|)]^{- \left(\frac{1}{p}-\frac{1}{s^{\prime}}+\frac{1}{d} \right)},\text{ if } t< 0.
\end{cases} $$
\begin{remark}
If we consider the Navier-Stokes equations on Einstein manifolds, then the space $\mathcal{X}$ and $\mathbb{X}$ reduce to a simple space $C_b(\r_+,L^p(\Gamma(T\mathcal{M})))$ and $C_b(\r,L^p(\Gamma(T\mathcal{M})))$, respectively (see \cite{Xuan2022}). 
\end{remark}

We now state our first main result on the existence of mild solutions to Stokes equations \eqref{CCauchyStokes} in the following theorem: 
\begin{theorem}\label{th1} 	Let $(\mathcal{M},g)$ be a $d$-dimensional non-compact manifold  with negative Ricci curvature tensor and $p>d$.	
	\begin{itemize}
		\item[(i)]  Suppose that $v\in \mathbb{X}$ and the external force $f \in C_b(\r,L^p(\Gamma(T\cal M))\cap L^2(\Gamma(T\cal M)))$ are almost periodic functions.
		Then, problem \eqref{CCauchyStokes} (without the initial data $u_0$) on the full line time-axis has one and only one  almost periodic mild solution $\hat{u}\in \mathbb{X}$ satisfying 
	\begin{equation*}\label{eesper}
			\|\hat{u}\|_{\mathbb{X} } \le  \widetilde{C}\|f \|_{\infty,L^p\cap L^2} + \widetilde{M} \| v\|^2_{\mathbb X}.
	\end{equation*}
		\item[(ii)] Suppose that $v\in \mathcal{X}$ and the external force $f \in C_b(\r_+ ,L^p(\Gamma(T\cal M))\cap L^2(\Gamma(T\cal M)))$ are asymptotic almost periodic functions.
		Then, problem \eqref{CCauchyStokes} has one and only one  asymptotic almost periodic mild solution $\hat{u}\in \mathcal{X}$ satisfying 
		\begin{equation*}\label{eesper}
			\|\hat{u}\|_{\mathcal{X}} \le \widetilde{C} \left(\|u_0 \|_{L^p\cap L^2} + \|f \|_{\infty,L^p\cap L^2}\right) + \widetilde{M} \| v\|^2_{\cal X}.
		\end{equation*}
	\end{itemize}
\end{theorem} 

Analogously to the case of Stokes equations,  by a \emph{mild solution} 
to equation \eqref{DDivNavierStokes} we mean the vector-field-valued map  $u:\r_+\times \mathcal{M} \to \Gamma(T\mathcal M)$
satisfying  the integral equation  
\begin{equation}\label{MildS}
u(t) = e^{-t\mathcal{A}}u_0 + \int_0^te^{-(t-\tau)\mathcal{A}} 
 \mathbb{P}[(-\nabla_uu + f)(\tau)] d\tau\hbox{ for }t\ge 0.
\end{equation}
And by a \emph{mild solution} 
to equation \eqref{DDivNavierStokes} on the full time line, we mean the vector-field-valued map  $u:\r\times \mathcal{M} \to \Gamma(T\mathcal M)$
satisfying  the integral equation  
\begin{equation}\label{MildS}
	u(t) = \int_{-\infty}^te^{-(t-\tau)\mathcal{A}} 
	\mathbb{P}[(-\nabla_uu + f)(\tau)] d\tau\hbox{ for }t\in \mathbb R.
\end{equation}
We then state our second main result on the existence and uniqueness  of the periodic mild solution to \eqref{DDivNavierStokes}
 in the following theorem. %(see Theorem \ref{thm2.20} in Section \ref{S4}):
\begin{theorem}\label{th2}
Let $(\cal M,g)$ be a $d$-dimensional non-compact manifold with the negative Ricci curvature tensor and $p>d$.
\begin{itemize}
	\item[(i)] Assume that $f\in C_b(\r,L^p(\Gamma(T\cal M))\cap L^2(\Gamma(T\cal M)))$ be almost periodic. Then,  if $\|f\|_{\infty, L^p\cap L^2}$ is sufficiently small,  the equation \eqref{DDivNavierStokes} (without initial data $u_0$) has one and only one  almost periodic mild solution $\hat{u}$ on a small ball of  $\mathbb X$.
	\item[(ii)] Assume that $f\in C_b(\r_+,L^p(\Gamma(T\cal M))\cap L^2(\Gamma(T\cal M)))$ be asymptotically almost periodic. Then,  if $\|f\|_{\infty, L^p\cap L^2}$ is sufficiently small,  the equation \eqref{DDivNavierStokes} has one and only one asymptotically almost periodic mild solution $\hat{u}$ on a small ball of  $\cal X$.
\end{itemize}  

\end{theorem} 

Lastly, our third main result on the exponential stability of the almost periodic and asymptotically almost periodic mild solution is stated as follows: 
\begin{theorem}\label{th3} Assume that the external force $f(t)=0$. Then, the almost periodic (resp. asymptotically almost periodic) mild solution  $\hat{u}$ of equation \eqref{DDivNavierStokes} is exponentially stable in the sense that for any $t>1$,
\begin{equation*}\label{sstasol}
\|\hat{u}(t)\|^{\blacklozenge} \lesssim  e^{-\gamma t},
\end{equation*}
here we denote
$$\|\hat{u}(t)\|^{\blacklozenge} := \|\hat{u}(t)\|_2 + \|\hat{u}(t)\|_{L^p} + [c_d(t)]^{- \left(\frac{1}{p}-\frac{1}{\tilde{p}}+\frac{1}{d} \right)}\| \nabla \hat{u}(t)\|_{L^{\tilde p}} + [c_d(t)]^{-\left( \frac{1}{p} - \frac{1}{s}+\frac{1}{d} \right)} \| \nabla \hat{u}(t)\|_{L^s};$$
whereas $\gamma$ is a positive constant satisfying 
$0<\gamma<\beta$.
\end{theorem}

We will give the proofs of our main results in Section \ref{S3}.

\section{Proofs of the main results}\label{S3}

In this section, we will give the proofs of our three main results stated in Subsection \ref{main}. We first prove the existence of bounded mild solutions to linear inhomogeneous Stokes
 equations \eqref{CCauchyStokes} with bounded external forces. 
 To that purpose we need the $L^p-L^q$ - dispersive and smoothing estimates of the semigroup $e^{-t\A}$. These estimates have been proved for semigroups on noncompact manifolds satisfying $(H_1)-(H_4)$ by Pierfelice in \cite[Corollary 4.13 and Theorem 4.15]{Pi}.  We recall these dipsersive and smoothing estimates in the following lemma:
\begin{lemma}\label{estimates}
Assuming $(H_1)-(H_4)$, putting $c_d(t) = C_0\max \left( \frac{1}{t^{d/2}},1 \right)$. There exist $\beta\geq C_0>0$ and some $C>0$ such that the solution of \eqref{CauchyStokes0} satisfies the following dispersive and smoothing estimates
\begin{itemize}
\item[i)] For $2\leq p \leq r <+\infty$ and for all $u_0 \in L^p(\Gamma(T\mathcal{M})) \cap L^2(\Gamma(T\mathcal{M})),$ 
\begin{equation}\label{dispersive0}
\left\| e^{-t\mathcal{A}} u(t)\right\|_{L^r} \leqslant C[c_d(t)]^{\frac{1}{p}-\frac{1}{r}}e^{-\beta t} \left( \left\| u_0 \right\|_{L^p} + \left\| u_0 \right\|_{L^2} \right),\, \forall t>0,
\end{equation}
\begin{equation}\label{smoothing0}
\left\| \nabla e^{-t\mathcal{A}}u(t) \right\|_{L^r} \leqslant C[c_d(t)]^{\frac{1}{p}-\frac{1}{r}+\frac{1}{d}}e^{-\beta t} \left( \left\| u_0 \right\|_{L^p} + \left\| u_0 \right\|_{L^2} \right), \, \forall t>0. 
\end{equation}

\item[ii)] For $1 < p \leq 2 \leq r < +\infty$ and for all $u_0 \in L^p(\Gamma(T\mathcal{M})),$ 
\begin{equation}\label{dispersive}
\left\| e^{-t\mathcal{A}} u(t)\right\|_{L^r} \leqslant C[c_d(t)]^{\frac{1}{p}-\frac{1}{r}}e^{-\beta t}\left\| u_0 \right\|_{L^p}, \, \forall t>0 ,
\end{equation}
\begin{equation}\label{smoothing}
\left\| \nabla e^{-t\mathcal{A}}u(t) \right\|_{L^r} \leq C[c_d(t)]^{\frac{1}{p}-\frac{1}{r}+\frac{1}{d}}e^{-\beta t} \left\| u_0 \right\|_{L^p}, \, \forall t>0.
\end{equation}
\end{itemize}
\end{lemma}
 
The following lemma gives us the boundedness (in time) of mild solutions to \eqref{CCauchyStokes} for each bounded external force.
\begin{lemma}\label{Thm:linear} Consider Stokes equation \eqref{CCauchyStokes} on $d$-dimensional non-compact manifold $(\mathcal{M},g)$ with negative Ricci curvature tensor and $p>d$.
\begin{itemize}
\item[(i)]   Let $u_0 \in L^p (\Gamma (T\cal M))\cap L^2 (\Gamma (T\cal M))$ and suppose that $v \in \cal X$, $f\in C_b(\mathbb{R}_+, (L^p\cap L^2)(\Gamma(T\cal M)))$. 
Then, the problem \eqref{CCauchyStokes} has one and only one mild solution $u\in \cal X$ given by the formula \eqref{mild:linear} with $u(0)=u_0$. Furthermore, we have
\begin{equation}\label{rgl}
\| u \|_{\cal X} \leq \widetilde{C} (\| u_0\|_{L^p\cap L^2} + \| f\|_{\infty,L^p\cap L^2}) + \widetilde{M} \| v\|^2_{\cal X}
\end{equation}
for some constants $\widetilde{C}$ and $\widetilde{M}$ independent of $u_0,\ u,\, v$ and $f$.
\item[(ii)]   Let   $v \in \mathbb X$ and $f\in C_b(\mathbb{R}, L^p(\Gamma(T\cal M)))\cap L^2(\Gamma(T\cal M)))$. Then, the problem \eqref{CCauchyStokes} (without initial data $u_0$) on the full line time-axis has one and only one mild solution $u\in \mathbb X$ given by the formula \eqref{mild:linear}. Furthermore, we have
\begin{equation}\label{rgl}
\| u \|_{\mathbb X} \leq \widetilde{C}\|  f\|_{\infty,L^p\cap L^2} + \widetilde{M} \| v\|^2_{\mathbb X}
\end{equation}
for some constants $\widetilde{C}$ and $\widetilde{M}$ as Assertion (i).
\end{itemize}
 
\end{lemma}
\begin{proof}
\def\A{\mathcal{A}}
The proof of Assertion (i) is done in the recent work \cite[Lemma 3.2]{HuyXuan2022}. Therefore, we prove only Assertion (ii) which consists the boundedness of AAP- mild solution on the full line time-axis. For the sake of convenience, we denote $\|\cdot\|_{p}:=\|\cdot\|_{L^p(\Gamma(T\mathcal{M}))}$. Note that, on Riemannian non-compact manifold with negative curvature tensor we have the $L^p-$boundedness of Riesz transform (see \cite{Loho}), therefore the operator $\mathbb{P}$ is bounded. 
 
Using the first inequality in assertion $i)$ of Lemma \ref{estimates}, one has
\begin{eqnarray*}
\| u(t)\|_2 &\leqslant&   \int_{-\infty}^t \|e^{-(t-\tau)\mathcal{A}} \mathbb{P} [\nabla_vv+f] (\tau)\|_2 d\tau   \le   \int_{-\infty}^t e^{-\beta(t-\tau)} \| \nabla_v v(\tau) + f (\tau)\|_2d\tau\cr
&\leqslant&    C\int_{-\infty}^te^{-\beta (t-\tau)} \left( \| v(\tau)\|_p \| \nabla v(\tau)\|_{\tilde p} + \|f\|_2 \right) d\tau  \cr
&\leqslant&  C\int_{-\infty}^t e^{-\beta(t-\tau)}d\tau \| f\|_{\infty,L^2} + \int_{-\infty}^t e^{-\beta (t-\tau)}   \left( \tilde{\lambda}(\tau) \right)^{-1}  d\tau \| v\|^2_{\mathbb X}\cr 
&\leqslant&  \frac{C}{\beta}\| f\|_{\infty,L^2} + G_1(t)\norm{v}^2_{\mathbb X} \le  \frac{C}{\beta}\| f\|_{\infty,L^p\cap L^2} + M_1 \| v\|^2_{\mathbb X},
\end{eqnarray*}
here the constant $M_1$ is given by (see Appendix) 
$$G_1(t) := \int_{-\infty}^t e^{-\beta (t-\tau)}  \left( \tilde{\lambda}(\tau) \right)^{-1} d\tau\le M_1<+\infty.$$

By putting $\dfrac{1}{r}=\dfrac{1}{p}+ \dfrac{1}{s}$ and using again the first inequality in item $i)$ of Lemma \ref{estimates}, we have
\begin{eqnarray*}
\|u(t)\|_p &\le &  C\int_{-\infty}^t e^{-\beta(t-\tau)}\| f(\tau)\|_{L^p\cap L^2} d\tau \cr 
&+& \int_{-\infty}^t [c_d(t-\tau)]^{\frac{1}{r}-\frac{1}{p}}e^{-\beta(t-\tau)} (\| \nabla_vv(\tau)\|_r + \| \nabla_vv(\tau)\|_2)d\tau\cr
&\leq&   C\| f\|_{\infty,L^p \cap L^2}\int_{-\infty}^t e^{-\beta(t-\tau)} d\tau\cr
&&+  \int_{-\infty}^t [c_d(t-\tau)]^{\frac{1}{s}}e^{-\beta (t-\tau)} (\|v(\tau)\|_p\|\nabla v(\tau) \|_s + \| v(\tau)\|_p\| \nabla v(\tau)\|_{\tilde{p}}) d\tau \cr
&\le&  \frac{C}{\beta} \| f\|_{\infty,L^p\cap L^2}  + \int_{-\infty}^t [c_d(t-\tau)]^{\frac{1}{s}}\left(  \left[ \hat{ \lambda}(\tau)\right]^{-1} +\left[ \tilde{ \lambda}(\tau)\right]^{-1}  \right)e^{-\beta(t-\tau)} d\tau \| v \|^2_{\mathbb X}\cr
&\leq&   \frac{C}{\beta} \| f\|_{\infty,L^p\cap L^2} +  M_2 \| v \|^2_{\mathbb X},
\end{eqnarray*}
here, we use the boundedness of following integral (see Appendix)
$$G_2(t) := \int_{-\infty}^t [c_d(t-\tau)]^{\frac{1}{s}}\left(  \left[ \hat{ \lambda}(\tau)\right]^{-1} +\left[ \tilde{ \lambda}(\tau)\right]^{-1}  \right)e^{-\beta(t-\tau)} d\tau\le M_2<+\infty.$$
 
{Using the second inequality of Lemma \ref{estimates}, we obtain the estimates for $L^{\tilde p}$-norm of the covariant derivative $\nabla u(t)$ as follows:}
\begin{eqnarray*}
\tilde{ \lambda}(t)\|\nabla u(t)\|_{\tilde{p}} &\le&  \tilde{ \lambda}(t)\int_{-\infty}^t  [c_d(t-\tau)]^{\left( \frac{1}{p}-\frac{1}{\tilde p}+ \frac{1}{d}\right)}e^{-\beta(t-\tau)}(\| f(\tau)\|_2 + \| f(\tau)\|_p )d\tau\cr
	&+& \tilde{ \lambda}(t)\int_{-\infty}^t [c_d(t-\tau)]^{\frac{1}{r}-\frac{1}{\tilde p} + \frac{1}{d}}e^{-\beta(t-\tau)} \left(\| \nabla_vv(\tau)\|_r+\| \nabla_vv(\tau)\|_2\right) d\tau \cr
&\le& \tilde{ \lambda}(t)\int_{-\infty}^t  \left[ c_d(t-\tau)\right] ^{\left( \frac{1}{p}-\frac{1}{\tilde p}+ \frac{1}{d}\right)}e^{-\beta(t-\tau)}(\| f(\tau)\|_2 + \| f(\tau)\|_p )d\tau\cr
	&+&  \tilde{ \lambda}(t)\int_{-\infty}^t [c_d(t-\tau)]^{\frac{1}{r} -\frac{1}{\tilde{p}}+\frac{1}{d}}e^{-\beta(t-\tau)}\cr &&\hspace{5cm}(\|v(\tau)\|_p\|\nabla v(\tau) \|_s + \| v(\tau)\|_p\| \nabla v(\tau)\|_{\tilde{p}}) d\tau\cr
&\le& \tilde{ \lambda}(t)\int_{-\infty}^t  [c_d(t-\tau)]^{\left( \frac{1}{p}-\frac{1}{\tilde p}+ \frac{1}{d}\right)}e^{-\beta(t-\tau)}(\| f(\tau)\|_2 + \| f(\tau)\|_p )d\tau\cr
	&+& \tilde{ \lambda}(t)\int_{-\infty}^t [c_d(t-\tau)]^{\frac{1}{r}-\frac{1}{\tilde p} + \frac{1}{d}}e^{-\beta(t- \tau)}\left(  \left[ \hat{ \lambda}(\tau)\right]^{-1} +\left[ \tilde{ \lambda}(\tau)\right]^{-1}  \right) d\tau \| v \|^2_{\mathbb X} \cr
&\leq&   G_{31}(t)\| f\|_{\mathbb X} +  G_{32}(t) \| v \|^2_{\mathbb X}, 
\end{eqnarray*}
where  $G_{31}, G_{32}$ are bounded (see Appendix) with
\begin{eqnarray*}
&&G_{31}(t) := \tilde{ \lambda}(t)\int_{-\infty}^t [c_d(t-\tau)]^{\left( \frac{1}{p}-\frac{1}{\tilde p}+ \frac{1}{d}\right)}e^{-\beta(t-\tau)} d\tau\le M_{31}<+\infty,\cr
&&G_{32}(t) := \tilde{ \lambda}(t)\int_{-\infty}^t [c_d(t-\tau)]^{\frac{1}{r}-\frac{1}{\tilde p} + \frac{1}{d}}e^{-\beta(t- \tau)}\left(  \left[ \hat{ \lambda}(\tau)\right]^{-1} +\left[ \tilde{ \lambda}(\tau)\right]^{-1}  \right) d\tau\le M_{32}<+\infty.   
\end{eqnarray*}

Using again the second inequality in Lemma \ref{estimates}, we obtain the estimates for $L^s$-norm of the covariant derivative $\nabla u(t)$ as follows.
\begin{eqnarray*}
&&\hat{\lambda}(t)\|\nabla u(t)\|_{s} \le \hat{ \lambda}(t)  \int_{-\infty}^t  [c_d(t-\tau)]^{ \frac{1}{p}-\frac{1}{s} +\frac{1}{d}}e^{-\beta(t-\tau)}(\| f(\tau)\|_2 + \| f(\tau)\|_p )d\tau\cr
&+&\hat{ \lambda}(t) \int_{-\infty}^t  [c_d(t-\tau)]^{\frac{1}{r}-\frac{1}{s} + \frac{1}{d}}e^{-\beta(t-\tau)} \left(\| \nabla_vv(\tau)\|_r+\| \nabla_vv(\tau)\|_2\right) d\tau \cr
&\leq&   G_{41}(t)\| f\|_{\mathbb X} +\hat{ \lambda}(t) \int_{-\infty}^t  [c_d(t-\tau)]^{\frac{1}{s}+\frac{1}{d}}e^{-\beta(t-\tau)} (\|v(\tau)\|_p\|\nabla v(\tau) \|_s + \| v(\tau)\|_p\| \nabla v(\tau)\|_{\tilde{p}}) d\tau \cr
&\le&   G_{41}(t)\| f\|_{\mathbb X}  +\hat{ \lambda}(t) \int_{-\infty}^t  [c_d(t-\tau)]^{\frac{1}{s} + \frac{1}{d}}\left( \left[\hat{ \lambda}(\tau)\right] ^{-1} +\left[ \tilde{ \lambda}(\tau)\right] ^{-1} \right) e^{-\beta(t- \tau)} d\tau \| v \|^2_{\mathbb X}\cr
&\leq&   G_{41}(t)\| f\|_{\mathbb X}+ G_{42}(t) \| v\|^2_{\mathbb X},
\end{eqnarray*}
where  $G_{41}, G_{42}$ are bounded (see Appendix) with
\begin{eqnarray*}
	 &&G_{41}(t) := \hat{ \lambda}(t)\int_{-\infty}^t [c_d(t-\tau)]^{ \frac{1}{p}-\frac{1}{s} +\frac{1}{d}}e^{-\beta(t-\tau)} d\tau\le M_{41}<+\infty,\cr
	 &&G_{42}(t) :=\hat{ \lambda}(t) \int_{-\infty}^t [c_d(t-\tau)]^{\frac{1}{s} + \frac{1}{d}}\left( \left[\hat{ \lambda}(\tau)\right] ^{-1} +\left[ \tilde{ \lambda}(\tau)\right] ^{-1} \right) e^{-\beta(t- \tau)} d\tau \le M_{42}<+\infty. 
 \end{eqnarray*}
 
	Finally, the inequality \eqref{rgl} holds as desired if we take 
$$\widetilde{C} = \max\left\{ C_1,\, M_{31},\, M_{41} \right\} \hbox{  and } \widetilde{M} = \max\left\{ M_1,\,M_2,\, M_{32},\, M_{42}\right\}.$$
\end{proof}

We now prove our three main theorems.

\textbf{ \underline{\textit {Proof of Theorem \ref{th1}.}} } 
\noindent
Firstly, we denote
$$ \| u(t) \|^{\blacklozenge} = \| u(t)\|_{L^2} + \| u(t)\|_{L^p} + [c_d(t)]^{- \left(\frac{1}{p}-\frac{1}{\tilde{p}}+\frac{1}{d} \right)}\| \nabla u(t)\|_{L^{\tilde p}} + [c_d(t)]^{- \left(\frac{1}{p}-\frac{1}{s}+\frac{1}{d} \right)} \| \nabla u(t)\|_{L^s},$$
and 
$$ \| u (t)\|^{\lozenge} = \| u(t)\|_{L^p} +\|u(t)\|_{L^2}    + \tilde{\lambda}(t) \| u(t)\|_{L^{\tilde{p}}}+\hat{ \lambda}(t) \| u(t)\|_{L^{s}}.$$
(i) Setting $Y=L^p(\Gamma(T\mathcal{M}))\cap L^2(\Gamma(T\mathcal{M}))$. Since the existence of bounded mild solution of equation \eqref{CCauchyStokes} (without the initial data) on the full line time-axis is established in Lemma \ref{Thm:linear}, we can define a solution operator associated with equation \eqref{CCauchyStokes} as follows 
\begin{align*}
S: \mathbb{X}\times C_b(\r,Y) &\to \mathbb{X}\cr
(v,f) &\mapsto u,
\end{align*}
where $u(t)$ is the bounded mild solution of equation \eqref{CCauchyStokes} on the full line time-axis. This means that 
\begin{equation}
S(v,f)(t) = u(t)
\end{equation}
for all $t\in \mathbb{R}$.

Indeed, since $ (v,f)\in \mathbb{X}\times C_b(\r,Y)$ is  the asymptotic periodic function, for each $ \epsilon  > 0$, there exists $l_{\epsilon}>0 $ such that every interval of length $l_{\epsilon}$ contains at least a number $T $ with the following property
$$\sup_{t \in \r } \{ \| v(t+T)  - v(t) \|^{\lozenge} + \| f(t+T) - f(t) \|_Y \} < \epsilon.$$

By the same way as in the proof of Assertion ii) of Lemma \ref{Thm:linear}, we can estimate
\begin{eqnarray}\label{AP1}
&&\left\|{S}(v,f)(t+T) - {S}(v,f)(t)\right\|^{\lozenge} \cr
&=&\left\|\int_{-\infty}^{t+T} e^{-(t+T-\tau)\mathcal{A}}\mathbb{P}[-\nabla_v v +f](\tau) d\tau - \int_{-\infty}^{t} e^{-(t-\tau)\mathcal{A}}\mathbb{P}[-\nabla_v v + f](\tau) d\tau \right\|^{\lozenge} \cr
&=&\left\|\int_{0}^\infty e^{-\tau\mathcal{A}}\mathbb{P}[-\nabla_v v(t+T-\tau) + f(t+T-\tau) + \nabla_vv(t-\tau) - f(t-\tau)] d\tau \right\|^{\lozenge} \cr
&=&\left\|\int_{0}^\infty e^{-\tau\mathcal{A}}\mathbb{P}\left[-\nabla_{v(t+T-\tau)-v(t-\tau)} v(t+T-\tau) + \nabla_{v(t-\tau)}(v(t-\tau)-v(t+T-\tau))\right]d\tau\right\|^{\lozenge}  
\cr
&+ &\left\|\int_{0}^\infty e^{-\tau\mathcal{A}}\mathbb{P}\left[ f(t+T-\tau) - f(t-\tau)\right] d\tau \right\|^{\lozenge} \cr
&\le& \widetilde{C} \left(\| v(\cdot+T)- v(\cdot)\|_{\mathbb{X}} + \| f(\cdot+T) - f(\cdot) \|_{C_b(\r,Y)}\right)  < 2\widetilde{C}\epsilon,
\end{eqnarray}
for all $t \in \r$, where $\widetilde{C}$ is determined as in Assertion ii) of Lemma \ref{Thm:linear}. 
Therefore, solution operator $S$ maps $(v,f)$ to an almost periodic function in $\mathbb{X}$. This shows the existence of almost periodic mild solution of equation \eqref{CCauchyStokes}.

(ii) Since the existence of bounded mild solution of equation \eqref{CCauchyStokes} (with the initial data $u_0$) on the half line time-axis is established in Assertion (ii) of Lemma \ref{Thm:linear}, we can define a solution operator associated with equation \eqref{CCauchyStokes} as follows 
\begin{align*}
\mathcal{S}: \mathcal{X}\times C_b(\r_+,Y) &\to \mathcal{X}\cr
(v,f) &\mapsto u,
\end{align*}
where $u(t)$ is the bounded mild solution of equation \eqref{CCauchyStokes} on the half line time-axis. This means that 
\begin{equation}
\mathcal{S}(v,f)(t) = u(t)
\end{equation}
for all $t\in \mathbb{R}_+$.

In fact, for each $ (v,f)\in \cal{X}\times C_b(\r_+,Y)$ is  the asymptotic almost periodic function, there exists the asymptotic periodic function $(\eta,h)\in \cal{X}\times C_b(\r_+,Y)$ and $(\omega,\phi) \in \cal X\times C
_0(\r_+,  Y)$ with $\lim_{t\to +\infty}\|\omega(t)\|^{\blacklozenge}_{\cal X}=0$ such that $v(t)=\eta(t)+ \omega(t)$ and $f(t)=h(t)+\phi(t)$ for all $t\in \r_+$. Therefore, for all $t\in\r_+$, one has
\begin{eqnarray*}
S(v,f)(t)&=& e^{-t\cal A}u(0)+\int_0^t e^{-(t-\tau)\cal A}\mathbb P[-\nabla_v v +f](\tau) d\tau \cr
&=& e^{-t\cal A}u(0)+\int_0^t e^{-(t-\tau)\cal A}\mathbb P \, h(\tau) d\tau +\int_0^t e^{-(t-\tau)\cal A}\mathbb P \, \phi(\tau) d\tau\cr 
&-& \int_0^t e^{-(t-\tau)\cal A}\mathbb P\,[\nabla_{\eta(\tau)} \eta(\tau)] d\tau-\int_0^t e^{-(t-\tau)\cal A}\mathbb P\,[\nabla_{\omega(\tau)} v(\tau)] d\tau \cr 
&-&\int_0^t e^{-(t-\tau)\cal A}\mathbb P\,[\nabla_{\eta(\tau)}\omega(\tau)] d\tau
\cr
&=& \int_{-\infty}^t e^{-(t-\tau)\cal A}\mathbb P\,h(\tau) d\tau -\int_{-\infty}^t e^{-(t-\tau)\cal A}\mathbb P\,[\nabla_{\eta(\tau)}\eta(\tau)] d\tau\cr
&+&e^{-t\cal A}u(0)  -\int_{-\infty}^0 e^{-(t-\tau)\cal A}\mathbb P\, h(\tau) d\tau +\int_{-\infty}^0 e^{-(t-\tau)\cal A}\mathbb P\,[\nabla_{\eta(\tau)}\eta(\tau)] d\tau \cr 
&+&\int_0^t e^{-(t-\tau)\cal A}\mathbb P\, \phi(\tau) d\tau  -\int_0^t e^{-(t-\tau)\cal A}\mathbb P\,[\nabla_{\omega(\tau)}v(\tau)] d\tau \cr 
&-&\int_0^t e^{-(t-\tau)\cal A}\mathbb P\,[\nabla_{\eta(\tau)}\omega(\tau)] d\tau.
\end{eqnarray*}
We set
\begin{eqnarray*}
&&S_a(h)(t)=   \int_{-\infty}^t e^{-(t-\tau)\cal A}\mathbb P\, h(\tau) d\tau; \;
S_a(\nabla_{\eta}\eta)(t)=\int_{-\infty}^t e^{-(t-\tau)\cal A}\mathbb P\,[\nabla_{\eta(\tau)}\eta(\tau)] d\tau;
\cr
&&S_0(\phi)(t)= \int_0^t e^{-(t-\tau)\cal A}\mathbb P\,\phi(\tau) d\tau; \,\;\;S_0(\nabla_{\omega}v)(t)= \int_0^t e^{-(t-\tau)\cal A}\mathbb P\,[\nabla_{\omega(\tau)} v(\tau)] d\tau; \cr 
&&S_0(\nabla_{\eta}\omega)(t)=\int_0^t e^{-(t-\tau)\cal A}\mathbb P\,[\nabla_{\eta(\tau)}\omega(\tau)] d\tau.
\end{eqnarray*}
By the same way as in the proof of Lemma \ref{Thm:linear}, it is easy to justify that $S_a(h)$ and $S_a(\nabla_{\eta}\eta)$ are bounded functions in $\mathbb X$.
We now show that $S(v,f)$ is an asymptotic almost function by three following claims:

{\bf Claim 1.} We prove that $S_a(h) - S_a(\nabla_{\eta}\eta)$ is an almost periodic function. In fact, by the hypothesis  $(\eta,h)$ is the asymptotic function, for each $\varepsilon>0$ such that there is $L_{\varepsilon}>0$ every interval with length $L_{\varepsilon}$ contains at least a number $T$ so that
$$\sup_{t\in \r}(\|\eta(t+T)-\eta(t)\|^{\blacklozenge}_{ \cal X}+ \|h(t+T)-h(t)\|_{Y}  )<\varepsilon   $$

By the same proof of Assertion ii) of Lemma \ref{Thm:linear}, one has

\begin{align*}
&\left\|S_a(h)(t+T)-S_a(h)(t)\right\|_Y\cr
&=\left\| \int_{-\infty}^{t+T} e^{-(t+T-\tau)\cal A}\mathbb P\, h(\tau) d\tau -\int_{-\infty}^{t} e^{-(t-\tau)\cal A}\mathbb P\, h(\tau) d\tau   \right\|_Y\cr
&=\left\| \int_0^{+\infty} e^{-z\cal A}\mathbb P[  h(t+T-z)-  h(t-z)] dz \right\|_Y\cr
&\le \tilde{C}\sup_{z\in \r}\|h(z+T)-h(z)\|_Y<\tilde{C}\varepsilon,\, \forall t\in \r,
\end{align*}
here $\tilde{C}$ is determined in Lemma \ref{Thm:linear}. Furthermore, 
\begin{eqnarray*}
&&\left\|S_a(\nabla_{\eta}\eta)(t+T)-S_a(\nabla_{\eta}\eta)(t)\right\|^{\blacklozenge}_{\cal X}\cr
&=&\left\| \int_{-\infty}^{t+T} e^{-(t+T-\tau)\cal A}\mathbb P[\nabla_{\eta(\tau)}\eta(\tau)] d\tau -\int_{-\infty}^{t} e^{-(t-\tau)\cal A}\mathbb P[\nabla_{\eta(\tau)}\eta(\tau)] d\tau   \right\|^{\blacklozenge}_{\cal X}\cr
&\le&\left\| \int_0^{+\infty} e^{-z\cal A}\mathbb P [\nabla_{\eta(t+T-z)}\eta(t+T-z)]d\tau-\int_0^{+\infty} e^{-z\cal A}\mathbb P [\nabla_{\eta(t-z)}\eta(t-z)]d\tau \right\|^{\blacklozenge}_{\cal X}\cr
&\le& 2\tilde{M}\|\eta(\cdot+T)-\eta(\cdot)\|^{\blacklozenge}_{\cal X}\|\eta\|_{\cal X}<2\tilde{M}\|\eta\|_{\cal X}\varepsilon,\, \forall t\in \r,
\end{eqnarray*}
here $\tilde{M}$ is determined in Lemma \ref{Thm:linear}.
These pointed out that $S_a(h)-S_a(\nabla_{\eta}\eta)$ is the asymptotic function as claimed. 

{\bf Claim 2.}  $S_0(\phi)-S_0(\nabla_{\omega} v)-S_0(\nabla_{\eta}\omega)$ is the forward asymptotic and continuous function. Indeed, for the first term
\begin{eqnarray*}
S_0(\phi)(t)&=& \int_0^t e^{-(t-\tau)\cal A}\mathbb P\, \phi(\tau) d\tau
= \int_0^{\frac{t}{2}} e^{-(t-\tau)\cal A}\mathbb P\, \phi(\tau) d\tau + \int_{\frac{t}{2}}^t e^{-(t-\tau)\cal A}\mathbb P\,\phi(\tau) d\tau\cr 
&=& S_1(\phi)(t)+S_2(\phi)(t).
\end{eqnarray*}
For $t>2$, it is clear that
\begin{eqnarray*}
\|S_1(\phi)(t)\|_{L^p\cap L^2} &\lesssim& \int_0^{\frac{t}{2}} e^{-\beta(t-\tau)}d\tau \|\phi\|_{\infty, L^p\cap L^2} 
=\|\phi\|_{\infty, L^p\cap L^2}\dfrac{1}{\beta}\left[e^{-\frac{\beta t}{2}}-e^{-\beta t} \right].    
\end{eqnarray*}
Furthermore,
\begin{eqnarray*}
\tilde{\lambda}(t)\|\nabla S_1(\phi)(t)\|_{\tilde{p}} &\lesssim&\tilde{\lambda}(t) \int_0^{\frac{t}{2}} e^{-\beta(t-\tau)}d\tau\|\phi\|_{\infty, L^p\cap L^2}
 =\dfrac{1}{\beta}\left[e^{-\frac{\beta t}{2}}-e^{-\beta t} \right]  \|\phi\|_{\infty, L^p\cap L^2}.
\end{eqnarray*}
\begin{eqnarray*}
\hat{\lambda}(t)\|\nabla S_1(\phi)(t)\|_{s} &\lesssim&\hat{\lambda}(t) \int_0^{\frac{t}{2}} e^{-\beta(t-\tau)}d\tau\|\phi\|_{\infty, L^p\cap L^2}=
\dfrac{1}{\beta}\left[e^{-\frac{\beta t}{2}}-e^{-\beta t} \right]  \|\phi\|_{\infty, L^p\cap L^2}.
\end{eqnarray*}
These lead to $$\lim_{t\to +\infty} \|S_1(\phi)(t)\|^{\blacklozenge}_{\cal X}=0.$$

It is clear that for any $\varepsilon>0,$ one has $\|\phi(t)\|_Y<\varepsilon \;\forall t>t_0,$ for some $t_0>0$. Therefore, we imply
$$\|S_2(\phi)(t)\|^{\blacklozenge}_{\cal X} \le N \varepsilon,\, \;\forall t>t_0, \text{ for some }t_0>0.$$
 By two above reasons, we get 
 $$\lim_{t\to +\infty} \|S_0(\phi)(t)\|^{\blacklozenge}_{\cal X}=0.$$

There are similar for proving  $S_0(\nabla_{\omega} v)$ and $S_0(\nabla_{\eta}\omega)$ are the forward asymptotic and continuous functions, we merely justify for the case of $S_0(\nabla_{\omega} v)$. Indeed, one has 
\begin{eqnarray*}
S_0(\nabla_{\omega}v)(t)&=& \int_0^t e^{-(t-\tau)\cal A}\mathbb P\,[\nabla_{\omega(\tau)} v(\tau)] d\tau\cr
&=&\int_0^{t/2} e^{-(t-\tau)\cal A}\mathbb P\,[\nabla_{\omega(\tau)} v(\tau)] d\tau+\int_{t/2}^t e^{-(t-\tau)\cal A}\mathbb P\,[\nabla_{\omega(\tau)} v(\tau)] d\tau\cr
&=&S_3(\nabla_{\omega}v)(t)+S_4(\nabla_{\omega}v)(t).
\end{eqnarray*} 
For $t>2$, by  estimating analogously as proof of Lemma \ref{Thm:linear}. i), we get 
\begin{eqnarray*}
\|S_3(\nabla_{\omega}v)(t)\|_2&=& \int_0^{t/2}\| e^{-(t-\tau)\cal A}\mathbb P\,[\nabla_{\omega(\tau)} v(\tau)] \|_2d\tau\cr
&=& \int_0^{t/2} e^{-\beta(t-\tau)}\|\omega(\tau)\|_p \|\nabla v(\tau)\|_{\tilde{p}}d\tau
= \int_0^{t/2} e^{-\beta(t-\tau)}\left(\tilde{\lambda}(\tau)\right)^{-1}  d\tau\|v\|^2_{\mathbb X}   \cr
&\le& \int_0^1 e^{\beta(1-t)} \tau^{  -\frac{d}{2p}+\frac{d}{2\tilde p}-\frac{1}{2}}  d\tau\|v\|^2_{\mathbb X} + \int_1^{t/2} e^{-\beta(t-\tau)}  d\tau\|v\|^2_{\mathbb X}  \cr
&\le&\left[ \dfrac{e^{\beta(1-t)}}{\theta_{\tilde{p}}}+ \dfrac{e^{- \frac{\beta t}{2}}}{\beta} -\dfrac{ e^{\beta(1- t)} }{\beta}  \right] \|v\|^2_{\mathbb X} 
\end{eqnarray*} 
converges to $0$ as $t \to \infty.$ Moreover, we see that
\begin{eqnarray*}
\|S_3(\nabla_{\omega}v)(t)\|_p&=& \int_0^{t/2}\| e^{-(t-\tau)\cal A}\mathbb P\,[\nabla_{\omega(\tau)} v(\tau)] \|_pd\tau\cr
&=& \int_0^{t/2} [c_d(t-\tau)]^{\frac{1}{r}-\frac{1}{p}}e^{-\beta(t-\tau)}(\|\nabla_{\omega(\tau)} v(\tau)\|_r +\|\nabla_{\omega(\tau)} v(\tau)\|_2)d\tau\cr
&=& \int_0^{t/2} e^{-\beta(t-\tau)}(\|\omega(\tau)\|_p \|\nabla v(\tau)\|_s+\|\omega(\tau)\|_p \|\nabla v(\tau)\|_{\tilde{p}})d\tau\cr
&=& \int_0^{t/2} e^{-\beta(t-\tau)}\left[\left(\hat{\lambda}(\tau)\right)^{-1}+ \left(\tilde{\lambda}(\tau)\right)^{-1} \right]  d\tau\|v\|^2_{\mathbb X}   \cr
&\le& \int_0^1 e^{\beta(1-t)}\left[\tau^{  -\frac{d}{2p}+\frac{d}{2s}-\frac{1}{2}} +\tau^{  -\frac{d}{2p}+\frac{d}{2\tilde p}-\frac{1}{2}} \right] d\tau\|v\|^2_{\mathbb X} + \int_1^{t/2} 2e^{-\beta(t-\tau)}  d\tau\|v\|^2_{\mathbb X}  \cr
&\le&\left[ \dfrac{e^{\beta(1-t)}}{\theta_{s}}+\dfrac{e^{\beta(1-t)}}{\theta_{\tilde{p}}}+ \dfrac{2e^{- \frac{\beta t}{2}}}{\beta} -\dfrac{2e^{\beta(1- t)} }{\beta}  \right] \|v\|^2_{\mathbb X} 
\end{eqnarray*} 
converges to $0$ as $t \to \infty.$
Analogously, one also has
\begin{eqnarray*}
&\tilde{\lambda}(t)&\|\nabla S_3(\nabla_{\omega}v)(t)\|_{\tilde{p}}\cr 
&\le&\tilde{\lambda}(t) \int_0^{t/2} [c_d(t-\tau)]^{\frac{1}{r}-\frac{1}{\tilde p}+\frac{1}{d}}e^{-\beta(t-\tau)}(\|\nabla_{\omega(\tau)} v(\tau)\|_r +\|\nabla_{\omega(\tau)} v(\tau)\|_2)d\tau\cr
&\le& \int_0^{t/2} e^{-\beta(t-\tau)}(\|\omega(\tau)\|_p \|\nabla v(\tau)\|_s+\|\omega(\tau)\|_p \|\nabla v(\tau)\|_{\tilde{p}})d\tau  
\end{eqnarray*}
converges to $0$ as $t \to \infty.$ And,
\begin{eqnarray*}
&\hat{\lambda}(t)&\|\nabla S_3(\nabla_{\omega}v)(t)\|_s\cr 
&\le&\hat{\lambda}(t) \int_0^{t/2} [c_d(t-\tau)]^{\frac{1}{r}-\frac{1}{s}+\frac{1}{d}}e^{-\beta(t-\tau)}(\|\nabla_{\omega(\tau)} v(\tau)\|_r +\|\nabla_{\omega(\tau)} v(\tau)\|_2)d\tau\cr
&\le& \int_0^{t/2} e^{-\beta(t-\tau)}(\|\omega(\tau)\|_p \|\nabla v(\tau)\|_s+\|\omega(\tau)\|_p \|\nabla v(\tau)\|_{\tilde{p}})d\tau  
\end{eqnarray*}
converges to $0$ as $t \to \infty.$

These estimates lead to
$$\lim_{t\to +\infty}\| S_3(\nabla_{\omega}v)(t)\|_{\cal X}^{\blacklozenge}=0.$$

Moreover, since $\lim_{t\to +\infty}\|\omega(t)\|_{\cal X}{\blacklozenge} = 0,$ $\forall \varepsilon >0,$ we have $\|\omega(t)\|_{\cal X}^{\blacklozenge}<\varepsilon, \; \forall t>t_0$ for some $t_0>0$. Hence, it is not hard to show that
$$\| S_4(\nabla_{\omega}v)(t)\|_{\cal X}^{\blacklozenge} \leqslant N \varepsilon \text{ for } t>t_0.$$
That means 
$$\lim_{t\to +\infty}\| S_4(\nabla_{\omega}v)(t)\|_{\cal X}^{\blacklozenge} = 0 \text{ and then } \lim_{t\to +\infty}\| S_0(\nabla_{\omega}v)(t)\|_{\cal X}^{\blacklozenge} = 0.$$

{\bf Claim 3.}  $e^{-t\cal A}[ u(0)-S_a(h)(0)+S_a(\nabla_{\eta}\eta)(0)] $ is the forward asymptotic and continuous function. Indeed, Thanks to Lemma \ref{Thm:linear}. i) again, one gets
\begin{eqnarray*}
&&\|e^{-t\cal A}[ u(0)-S_a(h)(0)+S_a(\nabla_{\eta}\eta)(0)]\|_2\cr
&& \lesssim e^{-\beta_2 t}\Big\{\|u(0)\|_2+\|S_a(h)(0)\|_2+\|S_a(\nabla_{\eta}\eta)(0)\|_2\Big\}\cr
&& \lesssim e^{-\beta_2 t}\Big\{\|u(0)\|_2+\|h\|_{C_b(\r_+,Y)}+\|\eta\|_{\cal X}^2\Big\}\to 0 \text{ as } t\to +\infty.
\end{eqnarray*}
It also is clear that
\begin{eqnarray*}
&&\|e^{-t\cal A}[ u(0)-S_a(h)(0)+S_a(\nabla_{\eta}\eta)(0)]\|_p\cr
&& \lesssim e^{-\hat\beta_2 t}\Big\{\|u(0)\|_p+\|S_a(h)(0)\|_p+\|S_a(\nabla_{\eta}\eta)(0)\|_p\Big\}\cr
&& \lesssim e^{-\hat\beta_2 t}\Big\{\|u(0)\|_p+\|h\|_{C_b(\r_+,Y)}+\|\eta\|_{\cal X}^2\Big\}\to 0 \text{ as } t\to +\infty.
\end{eqnarray*}
here $\beta_2=d-1+\gamma_{2,2},\; \hat\beta_2=d-1+\gamma_{p,p}$. Moreover, 
\begin{eqnarray*}
&&\tilde{\lambda}(t)\|\nabla e^{-t\cal A}[ u(0)-S_a(h)(0)+S_a(\nabla_{\eta}\eta)(0)]\|_{\tilde{p}}\cr
&& \lesssim\tilde{\lambda}(t) [c_d(t)]^{\frac{1}{r}-\frac{1}{\tilde p}+\frac{1}{d}}e^{-\beta t} \Big\{ \|u(0)\|_r+\|S_a(h)(0)\|_r+\|S_a(\nabla_{\eta}\eta)(0)\|_r \cr
&&\hspace*{4.1cm}+ \|u(0)\|_2+\|S_a(h)(0)\|_2+\|S_a(\nabla_{\eta}\eta)(0)\|_2\Big\} \cr
&& \lesssim e^{-\beta t}\Big\{\|u(0)\|_r+ \|u(0)\|_2+2\|h\|_{C_b(\r_+,Y)}+2\|\eta\|_{\cal X}^2\Big\}\to 0 \text{ as } t\to +\infty.
\end{eqnarray*}
On the other hand,
\begin{eqnarray*}
&&\hat{\lambda}(t)\|\nabla e^{-t\cal A}[ u(0)-S_a(h)(0)+S_a(\nabla_{\eta}\eta)(0)]\|_s\cr
&& \lesssim\hat{\lambda}(t) [c_d(t)]^{\frac{1}{r}-\frac{1}{s}+\frac{1}{d}}e^{-\beta t}  \Big\{\|u(0)\|_r+\|S_a(h)(0)\|_r+\|S_a(\nabla_{\eta}\eta)(0)\|_r  \cr
&& \hspace*{4.1cm}+ \|u(0)\|_2+\|S_a(h)(0)\|_2+\|S_a(\nabla_{\eta}\eta)(0)\|_2\Big\}\cr
&& \lesssim e^{-\beta t}\Big\{\|u(0)\|_r+ \|u(0)\|_2+2\|h\|_{C_b(\r_+,Y)}+2\|\eta\|_{\cal X}^2\Big\}\to 0 \text{ as } t\to +\infty.
\end{eqnarray*}
Therefore, our proof is completed as desired.

We now use Theorem \ref{th1} to prove Theorem \ref{th2} as follows. 

\textbf{\underline{\textit {Proof of Theorem \ref{th2}.}} } 
i) We will study the existence of almost periodic mild solution of Equation \eqref{DDivNavierStokes} on the ball with center zero and radius $\rho.$
\begin{eqnarray*}
\mathcal B_{\rho}^{AP}=\Big\{v\in \mathbb X: v \text{ is the almost periodic function and }\|v\|_{\mathbb X}\le \rho\Big\}.
\end{eqnarray*}
By the Lemma \ref{Thm:linear}. ii), for each $v\in \mathcal B_{\rho}^{AP}$,  the linear equation 
\begin{equation}\label{solR}
 	u(t) =   \int_{-\infty}^te^{-(t-\tau)\mathcal{A}} \mathbb{P} [-\nabla_vv+f] (\tau) d\tau
 \end{equation}
has unique AAP- mild solution $u$ such that 
\begin{eqnarray}\label{lerho}
\|u\|_{\mathbb X}\le \tilde{C} \|f\|_{\infty, L^p\cap L^2}  +\tilde{M}\|v\|^2_{\mathbb X} \le \rho,
\end{eqnarray}
if $\|f\|_{\infty, L^p\cap L^2}$ and $\rho$ are small enough. And then, in this circumstance, one can define a map 
\begin{eqnarray*}
\Phi:\mathbb X &\to& \mathbb X.\cr
v&\mapsto& \Phi(v)=u
\end{eqnarray*}
That means the map $\Phi$ acts from $\mathcal B_{\rho}^{AP}$ into itself. Therefore, we are able to rewrite (\ref{solR}) as follows,
\begin{equation}\label{solPhi}
 	\Phi(v)(t) =   \int_{-\infty}^te^{-(t-\tau)\mathcal{A}} \mathbb{P} [-\nabla_vv+f] (\tau) d\tau.
 \end{equation}
In the case that $v_1,v_2 \in \mathcal B_{\rho}^{AP}$, it is clear that the function $u:= \Phi_1(v_1)-\Phi_2(v_2)$ is the unique almost periodic function to the equation
$$ \partial_tu-\cal Au=\mathbb P[-\nabla_{v_1}v_1+\nabla_{v_2}v_2] =\mathbb P[-(v_1-v_2)\cdot\nabla v_1- v_2\cdot \nabla(v_1-v_2)] $$
By using the same estimates of the proof Lemma 
\ref{Thm:linear} and the inequality (\ref{lerho}), we are able to imply that
\begin{eqnarray*} 
\|\Phi(v_1)-\Phi(v_2)\|_{\cal X}\le 2 \tilde{M} \rho \|v_1-v_2\|_{\cal X},
\end{eqnarray*}
if $\rho$ is sufficiently small. In this situation, this means that $\Phi$ is a contraction on $\mathcal B_{\rho}^{AP}$  and it has  the unique fixed point $\hat u$. Moreover, the function $\hat u$ is also an unique almost periodic -mild solution to Navier-Stokes Equation \eqref{DDivNavierStokes} in $\mathcal B_{\rho}^{AP}$.

\medskip
%We should also note that in the case $\mathcal B_{\rho}^{AP}$ replaced by $\cal B:= \{v\in \mathbb X: \|v\|_{\mathbb X}\le \rho\}$, one is able to see that for a external force $f\in C_b(\r,Y)$, if $\|f\|_{\infty, L^p\cap L^2}$ and $\rho$ are sufficiently small then the Navier-Stokes Equation \eqref{DDivNavierStokes} admits the unique bounded solution $\hat{u} \in \cal B_{\rho}$.

ii) We continue to investigate the existence of asymptotic almost periodic mild solution of Equation \eqref{DDivNavierStokes} on the ball with center zero and radius $\rho.$
\begin{eqnarray*}
\mathcal B_{\rho}^{AAP}=\Big\{v\in \cal X: v \text{ is the asymptotic almost periodic function and }\|v\|_{\cal X}\le \rho\Big\}.
\end{eqnarray*}
By the Lemma \ref{Thm:linear}. i), for each $v\in \mathcal B_{\rho}^{AAP}$,  the linear equation 
\begin{equation}\label{asolR}
 	u(t) = u(0)+  \int_0^te^{-(t-\tau)\mathcal{A}} \mathbb{P} [-\nabla_vv+f] (\tau) d\tau \hbox{ for }t\ge 0.
 \end{equation}
has unique AAP- mild solution $u$ such that 
\begin{eqnarray}\label{alerho}
\|u\|_{\cal X}\le \tilde{C}\left(\|u(0)\|_{L^p\cap L^2}+\|f\|_{\infty, L^p\cap L^2}\right) +\tilde{M}\|v\|^2_{\cal X} \le \rho,
\end{eqnarray}
if $\|u(0)\|_{L^p\cap L^2}, \|f\|_{\infty, L^p\cap L^2}$ and $\rho$ are small enough. In this situation, one can define a map 
\begin{eqnarray*}
\Phi:\cal X &\to& \cal X.\cr
v&\mapsto& \Phi(v)=u
\end{eqnarray*}
That means the map $\Phi$ acts from $\mathcal B_{\rho}^{AAP}$ into itself. Thus, we are able to rewrite (\ref{asolR}) as follows,
\begin{equation}\label{asolPhi}
 	\Phi(v)(t) = u(0)+  \int_0^te^{-(t-\tau)\mathcal{A}} \mathbb{P} [-\nabla_vv+f] (\tau) d\tau.
 \end{equation}
In the case that $v_1,v_2 \in \mathcal B_{\rho}^{AAP}$, it is clear that the function $u:= \Phi_1(v_1)-\Phi_2(v_2)$ is the unique asymptotic almost periodic function to the equation
$$ \partial_tu-\cal Au=\mathbb P[-\nabla_{v_1}v_1+\nabla_{v_2}v_2] =\mathbb P[-(v_1-v_2)\cdot\nabla v_1- v_2\cdot \nabla(v_1-v_2)] $$
By using the same estimates of the proof Lemma 
\ref{Thm:linear} and the inequality (\ref{alerho}), we are able to see that
\begin{eqnarray*} 
\|\Phi(v_1)-\Phi(v_2)\|_{\cal X}\le 2 \tilde{M} \rho \|v_1-v_2\|_{\cal X},
\end{eqnarray*}
if $\rho$ is sufficiently small. In this situation, this means that $\Phi$ is a contraction on $\mathcal B_{\rho}^{AAP}$  and it has  the unique fixed point $\hat u$. Moreover, the function $\hat u$ is also an unique asymptotic almost periodic -mild solution to Navier-Stokes Equation \eqref{DDivNavierStokes} in $\mathcal B_{\rho}^{AAP}$.

\medskip
%We should also note that in the case $\mathcal B_{\rho}^{AAP}$ replaced by $\cal B:= \{v\in \cal X: \|v\|_{\cal X}\le \rho\}$, one is able to see that for a external force $f\in C_b(\r_+,Y)$, if $\|f\|_{\infty, L^p\cap L^2}$ and $\rho$ are sufficiently small then the Navier-Stokes Equation \eqref{DDivNavierStokes} admits the unique bounded solution $\hat{u} \in \cal B_{\rho}$.

Now we use Gronwall's inequality to prove Theorem \ref{th3}

 \textbf{\underline{\textit {Proof of Theorem \ref{th3}.}}} 
 By Lemma \ref{estimates}, one has the following estimates
 \begin{eqnarray*}
 \|\hat u(t)\|_p &\le& e^{-\beta t}\|\hat u(0)\|_p+\int_0^t[c_d(t-\tau)]^{\frac{1}{r}-\frac{1}{p}}e^{-\beta (t-\tau)}(\|\nabla _{\hat u}\hat u(\tau)\|_r+\|\nabla _{\hat u}\hat u(\tau)\|_2)d\tau\cr
 &\le& e^{-\beta t}\|\hat u(0)\|_p+\int_0^t[c_d(t-\tau)]^{\frac{1}{s}}e^{-\beta (t-\tau)}[\|\hat u(\tau)\|_p\|\nabla \hat u(\tau)\|_s+\|\hat u(\tau)\|_p\|\nabla \hat u(\tau)\|_{\tilde p}]d\tau\cr
  &\le& e^{-\beta t}\|\hat u(0)\|_p+\|\hat u\|_{\cal X}\int_0^t[c_d(t-\tau)]^{\frac{1}{s}}e^{-\beta (t-\tau)} \left[\left(\hat \lambda(\tau)\right)^{-1}+\left( \tilde \lambda(\tau)\right)^{-1}\right] \|\hat u(\tau)\|_p   d\tau\cr
 \end{eqnarray*}
By putting $x(\tau)=e^{\gamma \tau}\|\hat u(\tau)\|_p$, for $\gamma <\beta$ we imply that
 \begin{eqnarray*}
 	x(t) &\le& \|\hat u(0)\|_p+\|\hat u\|_{\cal X}\int_0^t[c_d(t-\tau)]^{\frac{1}{s}}e^{-(\beta -\gamma)(t-\tau)} \left[\left(\hat \lambda(\tau)\right)^{-1}+\left( \tilde \lambda(\tau)\right)^{-1}\right] x(\tau)   d\tau,
 \end{eqnarray*}
 here it is easy to see that 
 $$\|\hat u\|_{\cal X}\int_0^t[c_d(t-\tau)]^{\frac{1}{s}}e^{-(\beta-\gamma) (t-\tau)} \left[\left(\hat \lambda(\tau)\right)^{-1}+\left( \tilde \lambda(\tau)\right)^{-1}\right]  d\tau\le M_5<+\infty.$$
 Therefore, we now use the Gronwall's inequality to get
 $$|x(t)|\le  \|\hat u(0)\|_p e^{M_5}  \text{ for all } t>0.$$
 This means that
  $$\|\hat u(t)\|_p\lesssim   e^{-\gamma t}  \text{ for all } t>0.$$
 Analogously, we are also able to verify that
   $$\|\hat u(t)\|_2\lesssim   e^{-\gamma t}  \text{ for all } t>0.$$
 
 In addition, by using again Lemma \ref{estimates}, we continue to give estimates for the covariant derivative $\nabla\hat u(t)$ as follows: for $t\ge 1$, we have
  \begin{eqnarray*}
 	&&e^{\gamma t}\left[ \tilde \lambda(t)\|\nabla \hat u(t)\|_{\tilde p}+\hat\lambda(t)\|\nabla \hat u(t)\|_s\right]=e^{\gamma t}\left[ \|\nabla \hat u(t)\|_{\tilde p}+\|\nabla \hat u(t)\|_s\right]  \cr
 	&\lesssim& 2e^{\gamma t}e^{-\beta t}\|\hat u(0)\|_p
 	+e^{\gamma t}\int_0^t\left[ [c_d(t-\tau)]^{\frac{1}{r}-\frac{1}{\tilde p}+\frac{1}{d}}+[c_d(t-\tau)]^{\frac{1}{r}-\frac{1}{s}+\frac{1}{d}}\right] e^{-\beta (t-\tau)}\cr 
 	&&\hspace{7.5cm}\times[\|\nabla _{\hat u}\hat u(\tau)\|_r+\|\nabla _{\hat u}\hat u(\tau)\|_2]d\tau \cr
 	&\lesssim& 2\|\hat u(0)\|_p
 	+e^{\gamma t}\int_0^t\left[ [c_d(t-\tau)]^{\frac{1}{r}-\frac{1}{\tilde p}+\frac{1}{d}}+[c_d(t-\tau)]^{\frac{1}{r}-\frac{1}{s}+\frac{1}{d}}\right] e^{-\beta (t-\tau)}\cr
 	&&\hspace{5.1cm}\times[\|\hat u(\tau)\|_p\|\nabla \hat u(\tau)\|_s+\|\hat u(\tau)\|_p\|\nabla \hat u(\tau)\|_{\tilde p}]d\tau\cr
&\lesssim& 2\|\hat u(0)\|_p
+\|\hat u\|_{\cal X}\int_0^t\left[ [c_d(t-\tau)]^{\frac{1}{r}-\frac{1}{\tilde p}+\frac{1}{d}}+[c_d(t-\tau)]^{\frac{1}{r}-\frac{1}{s}+\frac{1}{d}}\right] e^{-(\beta-\gamma) (t-\tau)}\cr
&&\hspace{3.5cm}\times\left(\hat\lambda(\tau)\right)^{-1} e^{\gamma \tau}\left[ \hat\lambda(\tau)\|\nabla \hat u(\tau)\|_s+\tilde\lambda(\tau)\|\nabla \hat u(\tau)\|_{\tilde p}\right] d\tau. 	 
 \end{eqnarray*}
 By the same way as Appendix, we also easily pointed out that $$ \|\hat u\|_{\cal X}\int_0^t\left[ [c_d(t-\tau)]^{\frac{1}{r}-\frac{1}{\tilde p}+\frac{1}{d}}+[c_d(t-\tau)]^{\frac{1}{r}-\frac{1}{s}+\frac{1}{d}}\right] e^{-(\beta-\gamma) (t-\tau)} \left( \hat \lambda(\tau)\right)^{-1}   d\tau\le M_6<+\infty.$$ 
Hence, we again use the Gronwall's inequality to get
 $$e^{\gamma t}\left[ \tilde \lambda(t)\|\nabla \hat u(t)\|_{\tilde p}+\hat\lambda(t)\|\nabla \hat u(t)\|_s\right]\le  2\|\hat u(0)\|_p e^{M_6}  \text{ for all } t\ge 1.$$
 This implies that 
  $$\tilde \lambda(t)\|\nabla \hat u(t)\|_{\tilde p}+\hat\lambda(t)\|\nabla \hat u(t)\|_s\lesssim e^{- \gamma t}    \text{ for all } t\ge 1.$$
By the above reasons, we are able to conclude that
$$\| \hat u(t)\|^{\blacklozenge} \lesssim e^{- \gamma t}    \text{ for all } t\ge 1.$$
\begin{remark}
In above three main theorems, we have proved the existence and exponential decay for AP and AAP-mild solutions for Stokes and Navier-Stokes equations in the phase spaces $\mathbb{X}$ and $\mathcal{X}$, respectively. These spaces consist the main component $C_b(\r,L^p(\Gamma(T\mathcal{M})))$ (for $p>d$). In fact, if we consider the Navier-Stokes equations on the real hyperbolic manifold $\mathbb{H}^d$, then the Kodaira-Hodge operator commutates with the vectorial heat semigroup and these phase spaces reduce to $C_b(\r, L^p(\Gamma(T\mathbb{H}^d)))$. Therefore, the results obtained in this paper are extension of the ones obtained in \cite{Xuan2022}.
\end{remark} 

\section{Appendix} 
\def\A{\mathcal{A}}
We now justify the boundedness of the integrals which are used in the previous sections. We first denote
 	$$\theta_{s^\prime}=\frac{1}{2}-\frac{d}{2p}+\frac{d}{2s^\prime};\;\,\theta_{\tilde p}=\frac{1}{2}-\frac{d}{2p}+\frac{d}{2\tilde{p}};\;\, \theta_{s}=\frac{1}{2}-\frac{d}{2p}+\frac{d}{2s};\;\, \theta_{r}=\frac{1}{2}-\frac{d}{2r}+\frac{d}{2\tilde{p}}.$$
  For the boundedness of 
 	\begin{eqnarray*}
 		G_1(t) := \int_{-\infty}^t e^{-\beta (t-\tau)}  \left( \tilde{\lambda}(\tau) \right)^{-1} d\tau.
 	\end{eqnarray*}

 	If $t\le -1$ then
 	
 	$$G_1(t)\lesssim \int_{-\infty}^t   [\max\{ {|\tau|^{-\frac{d}{2}},1}\}]^{\frac{1}{p}-\frac{1}{s^\prime}+\frac{1}{d}}e^{-\beta (t-\tau)}d\tau=\int_{-\infty}^{t}   e^{-\beta (t-\tau)}d\tau=\dfrac{1}{\beta}\le M_1.$$
 	
 	If $-1<t\le 0$ then
 	\begin{align*}
 		G_1(t) &\lesssim  
 		\int_{-\infty}^{-1}   e^{-\beta (t-\tau)}d\tau+\int_{-1}^t  |\tau|^{-\frac{d}{2}\left(\frac{1}{p}-\frac{1}{s^\prime}+\frac{1}{d}\right)}e^{-\beta (t-\tau)}d\tau \cr
 		&\leqslant \dfrac{1}{\beta}e^{-\beta(1+t)}+ \int_{-t}^1  \tau^{-\frac{d}{2}\left(\frac{1}{p}-\frac{1}{s^\prime}+\frac{1}{d}\right)}d\tau \le \dfrac{1}{\beta}+\dfrac{1-|t|^{\theta_{s^\prime}} }{\theta_{s^\prime}}\le M_1.
 	\end{align*}
 	
 	If $0<t\le 1$ then
 	\begin{align*}
 		G_1(t) &\lesssim  
 		\int_{-\infty}^{-1}   e^{-\beta (t-\tau)}d\tau+\int_{-1}^0  |\tau|^{-\frac{d}{2}\left(\frac{1}{p}-\frac{1}{s^\prime}+\frac{1}{d}\right)}e^{-\beta (t-\tau)}d\tau +\int_{0}^t  \tau^{-\frac{d}{2}\left(\frac{1}{p}-\frac{1}{\tilde{p}}+\frac{1}{d}\right)}e^{-\beta (t-\tau)}d\tau\cr
 		&\le  \dfrac{1}{\beta}e^{-\beta(1+t)}+ \int_0^1  \tau^{-\frac{d}{2}\left(\frac{1}{p}-\frac{1}{s^\prime}+\frac{1}{d}\right)}d\tau +\int_{0}^t  \tau^{-\frac{d}{2}\left(\frac{1}{p}-\frac{1}{\tilde{p}}+\frac{1}{d}\right)}d\tau
 		\le \dfrac{1}{\beta}+\dfrac{1}{\theta_{s^\prime}}+\dfrac{1}{\theta_{\tilde{p}}}+\le M_1.
 	\end{align*}
 	
 	If $1<t$ then
 	\begin{align*}
 		G_1(t) &\lesssim 
 		\int_{-\infty}^{-1}   e^{-\beta (t-\tau)}d\tau
 		+\int_{-1}^0  |\tau|^{-\frac{d}{2}\left(\frac{1}{p}-\frac{1}{s^\prime}+\frac{1}{d}\right)}d\tau  +\int_{0}^1  \tau^{-\frac{d}{2}\left(\frac{1}{p}-\frac{1}{\tilde{p}}+\frac{1}{d}\right)}d\tau+\int_1^t e^{-\beta (t-\tau)}d\tau
 		\\
 		&   \le \dfrac{1}{\beta} +\dfrac{1}{\theta_{s^\prime}}+\dfrac{1}{\theta_{\tilde{p}}}+ \dfrac{1}{\beta} := M_1.
 	\end{align*}
 Therefore, $G_1(t)\lesssim M_1<+\infty$ for all $t\in \mathbb R.$
 	
 	For the boundedness of 
 	 $$G_2(t) := \int_{-\infty}^t [c_d(t-\tau)]^{\frac{1}{s}}\left(  \left[ \hat{ \lambda}(\tau)\right]^{-1} +\left[ \tilde{ \lambda}(\tau)\right]^{-1}  \right)e^{-\beta(t-\tau)} d\tau.$$
 	 
 	In the case $t\leq-1$, 
 	\begin{eqnarray*}
 		G_2(t) &\lesssim&  \int_{-\infty}^t\left(  (t- \tau)^{-\frac{d}{2s}}+1\right)2 e^{-\beta(t-\tau)}d\tau \cr
 		&=&\dfrac{2}{\beta^{1-\frac{d}{2s}}} \int_{0}^{\infty} \tau^{-\frac{d}{2s}}e^{-\tau}d\tau +\dfrac{2}{\beta}=\dfrac{2}{\beta^{1-\frac{d}{2s}}} \Gamma\left(1-\frac{d}{2s} \right)  +\dfrac{2}{\beta}:= M_{2a}.
 	\end{eqnarray*}

 	In the case $-1<t\le 0$, 
 	\begin{eqnarray*}
 		G_2(t) &\lesssim& \int_{-\infty}^{-1}\left((t- \tau)^{-\frac{d}{2s}}+1\right)2 e^{-\beta(t-\tau)}d\tau    
 		+\int_{-1}^t   (t- \tau)^{-\frac{d}{2s}}  2|\tau|^{-\frac{d}{2}({\frac{1}{p}-\frac{1}{s^\prime}+\frac{1}{d}})}  d\tau\cr
 		&\le&\dfrac{2}{\beta^{1-\frac{d}{2s}}} \int_{0}^{\infty} z^{-\frac{d}{2s}}e^{-z}dz +\dfrac{2}{\beta}e^{-\beta(1+t)}   +\int_{0}^{t+1} z^{-\frac{d}{2s}}    2|t-z|^{-\frac{d}{2}({\frac{1}{p}-\frac{1}{s^\prime}+\frac{1}{d}})}   dz\cr
 		&\le&\dfrac{2}{\beta^{1-\frac{d}{2s}}} \Gamma\left(1-\frac{d}{2s} \right) +\dfrac{2}{\beta}    + \int_{0}^{t+1} z^{-\frac{d}{2s}}    2z^{-\frac{d}{2}({\frac{1}{p}-\frac{1}{s^\prime}+\frac{1}{d}})} dz \;\, (\hbox{by } |z-t|>z>0)
 		\cr
 		&\le& M_{2a}  + \dfrac{2}{\frac{1}{2}-\frac{d}{2r}+\frac{d}{2s^\prime}}:=M_{2b}.
 	\end{eqnarray*}
 	
 	In the case $0<t\le 1$, 
 	\begin{eqnarray*}
 		G_2(t) &\lesssim& \int_{-\infty}^{-1}2e^{-\beta(t-\tau)}d\tau   +\int_{-1}^0  \left(  (t- \tau)^{-\frac{d}{2s}} +1\right)  2|\tau|^{-\frac{d}{2}({\frac{1}{p}-\frac{1}{s^\prime}+\frac{1}{d}})}   d\tau\cr
 		&+&\int_0^t   (t- \tau)^{-\frac{d}{2s}}  \left[  \tau^{-\frac{d}{2}({\frac{1}{p}-\frac{1}{s}+\frac{1}{d}})} + \tau^{-\frac{d}{2}({\frac{1}{p}-\frac{1}{\tilde p}+\frac{1}{d}})} \right]  d\tau\cr
 		&\le& \dfrac{2}{\beta}e^{-\beta(1+t)}    +\int_{-1}^0  \left(  (- \tau)^{-\frac{d}{2s}} +1\right)  2|\tau|^{-\frac{d}{2}({\frac{1}{p}-\frac{1}{s^\prime}+\frac{1}{d}})}   d\tau
 		\cr
 		&+&\int_{0}^{t}\left( 1-\dfrac{ \tau}{t}\right) ^{-\frac{d}{2s}}  2\left( \dfrac{\tau}{t}\right) ^{-\frac{d}{2}({\frac{1}{p}-\frac{1}{s}+\frac{1}{d}})}.t^{-\frac{d}{2p}-\frac{1}{2}}   d\tau\;\, (\text{because } \tilde p<s)\cr 
 		&\le& \dfrac{2}{\beta}  +\dfrac{2}{\frac{1}{2}-\frac{d}{2r}+\frac{d}{2s^\prime}}   +\dfrac{2}{ \theta_{s^\prime}}    
 		+2\mathbf{B}\left(1-\frac{d}{2s}, \theta_s \right):=M_{2c}.
 	\end{eqnarray*}

 	In the case $t>1$,  
 	\begin{eqnarray*}
 G_2(t) &\lesssim& \int_{-\infty}^{-1}2e^{-\beta(t-\tau)}d\tau   +\int_{-1}^0 2|\tau|^{-\frac{d}{2}({\frac{1}{p}-\frac{1}{s^\prime}+\frac{1}{d}})} d\tau
 		+\int_0^{\frac{1}{2}}   \left( 2^{\frac{d}{2s}}+1\right)   2\tau^{-\frac{d}{2}({\frac{1}{p}-\frac{1}{s}+\frac{1}{d}})}  d\tau \cr
 		&+&\int_{\frac{1}{2}}^1   \left( (t- \tau)^{-\frac{d}{2s}}+1\right)   2.2^{\frac{d}{2}({\frac{1}{p}-\frac{1}{s}+\frac{1}{d}})} e^{-\beta(t-\tau)}  d\tau \; (\text{because } \tilde{p}<s)\cr
 		&+&\int_1^t   \left( (t- \tau)^{-\frac{d}{2s}}+1\right)   2e^{-\beta(t-\tau)}d\tau  \cr
 &\le& \dfrac{2}{\beta}e^{-\beta(1+t)}  +\dfrac{2}{\theta_s^\prime}    +\left( 2^{\frac{d}{2s}}+1\right)2. \dfrac{2^{-\theta_s}}{\theta_s}    +\left[2. 2^{1-\theta_s}+2\right]  \left(\dfrac{1}{\beta}\mathbf\Gamma\left(1-\dfrac{d}{2s} \right) +\dfrac{1}{\beta} \right)  \cr
 &\le& \dfrac{2}{\beta}  +\dfrac{2}{\theta_s^\prime}    +\left( 2^{\frac{d}{2s}}+1\right)2. \dfrac{2^{-\theta_s}}{\theta_s}    +\left[2. 2^{1-\theta_s}+2\right]  \left(\dfrac{1}{\beta}\mathbf\Gamma\left(1-\dfrac{d}{2s} \right) +\dfrac{1}{\beta} \right):=M_{2d}.
 	\end{eqnarray*}
 	Therefore,
 	$G_2(t)\lesssim M_2:= \max\left\{  M_{2b},\, M_{2c},\, M_{2d}\right\}<+\infty$ for all $t\in \mathbb R.$

    {We now estimate the integrals $G_{31},\, G_{32}$ and give the corresponding precise values of $M_{31}$ and $M_{32}$ as follows.}
 	\begin{eqnarray*}
 		G_{31}(t) &=& \tilde{ \lambda}(t)\int_{-\infty}^t [c_d(t-\tau)]^{\left( \frac{1}{p}-\frac{1}{\tilde p}+ \frac{1}{d}\right)}e^{-\beta(t-\tau)} d\tau,\cr
 		G_{32}(t) &=& \tilde{ \lambda}(t)\int_{-\infty}^t [c_d(t-\tau)]^{\frac{1}{r}-\frac{1}{\tilde p} + \frac{1}{d}}e^{-\beta(t- \tau)}\left(  \left[ \hat{ \lambda}(t)\right]^{-1} +\left[ \tilde{ \lambda}(t)\right]^{-1}  \right) d\tau.
 	\end{eqnarray*}

 	If $t\le -1$ then 
 	$$G_{31}(t) \lesssim \int_{-\infty}^t \left(  (t-\tau)^{{-\frac{d}{2}}\left( \frac{1}{p}-\frac{1}{\tilde p}+\frac{1}{d} \right)}+1\right) e^{-\beta(t-\tau)} d\tau =\dfrac{1 }{\beta^{\theta_{\tilde{p}}}}\mathbf \Gamma\left(\theta_{\tilde{p}} \right)  +\dfrac{1}{\beta}:=M_{31a}.$$

 	If $-1<t\le 0$ then 
 	\begin{eqnarray*}
 		G_{31}(t) &\lesssim&  \tilde{ \lambda}(t) \left[  \int_{-\infty}^{-1} \left(  (t-\tau)^{{-\frac{d}{2}}\left( \frac{1}{p}-\frac{1}{\tilde p}+\frac{1}{d} \right)}+1\right) e^{-\beta(t-\tau)} d\tau +
 		\int_{-1}^t   (t-\tau)^{{-\frac{d}{2}}\left( \frac{1}{p}-\frac{1}{\tilde p}+\frac{1}{d} \right)} d\tau  \right]\cr 
 		&\le&  \tilde{ \lambda}(t) \left[  \dfrac{1 }{\beta^{\theta_{\tilde{p}}}}\mathbf \Gamma\left(\theta_{\tilde{p}} \right)  +\dfrac{1}{\beta} +
 		\dfrac{(t+1)^{\theta_{\tilde{p}}}}{\theta_{\tilde{p}}}  \right]\cr
 		&\le&    \dfrac{1 }{\beta^{\theta_{\tilde{p}}}}\mathbf \Gamma\left(\theta_{\tilde{p}} \right)  +\dfrac{1}{\beta} +
 		\dfrac{1}{\theta_{\tilde{p}}} := M_{31b}.
 	\end{eqnarray*}

 	If $0<t \le 1$ then 
 	\begin{eqnarray*}
 		G_{31}(t) &\lesssim&  \tilde{ \lambda}(t)  \left[ \int_{-\infty}^{-1}  e^{-\beta(t-\tau)} d\tau  +  \int_{-1}^t \left(   (t-\tau)^{{-\frac{d}{2}}\left( \frac{1}{p}-\frac{1}{\tilde p}+\frac{1}{d} \right)} +1\right) d\tau  \right]\cr 
 		&\le&  \tilde{ \lambda}(t) \left[\dfrac{1}{\beta} +  
 		\dfrac{(t+1)^{\theta_{\tilde{p}}}}{\theta_{\tilde{p}}} +t+1 \right]\le \left[\dfrac{1}{\beta} +  
 		\dfrac{2^{\theta_{\tilde{p}}}}{\theta_{\tilde{p}}} +2 \right]:=M_{31c}.
 	\end{eqnarray*}

 	If $t > 1$ then 
 	\begin{eqnarray*}
 		G_{31}(t) &\lesssim& \int_{-\infty}^{-1}  e^{-\beta(t-\tau)} d\tau   +\int_{-1}^0  d\tau  +\int_{0}^t \left(   (t-\tau)^{{-\frac{d}{2}}\left( \frac{1}{p}-\frac{1}{\tilde p}+\frac{1}{d} \right)} +1\right)e^{-\beta(t-\tau)} d\tau \cr
 		&\le&  \left[\dfrac{1}{\beta} +1+  
 		\dfrac{1}{\beta^{\theta_{\tilde{p}}}}\mathbf\Gamma(\theta_{\tilde{p}})   +\dfrac{1}{\beta} \right]:=M_{31d}.
 	\end{eqnarray*}
 	Thus, $G_{31}(t)\lesssim M_{31}:=\max{\left\{M_{31b},M_{31c},M_{31d}\right\}}<+\infty$ for all $t\in \mathbb R.$

 	We continue to estimate  $G_{32}.$

 	If $t\le -1$ then
 	\begin{eqnarray*} 
 		G_{32}(t) &\lesssim& \int_{-\infty}^t   \left( (t-\tau)^{-\frac{d}{2}\left( {\frac{1}{r}-\frac{1}{\tilde p} + \frac{1}{d}}\right)}+1\right) 2e^{-\beta(t-\tau)} d\tau  =\dfrac{2}{\beta^{\theta_r}}\mathbf{\Gamma}\left( \theta_r\right) +\dfrac{2}{\beta}:= M_{32a}.
 	\end{eqnarray*}

 	If $-1<t<0$ then
 	\begin{eqnarray*} 
 		G_{32}(t) &\lesssim&  \int_{-\infty}^{-1}   \left( (t-\tau)^{-\frac{d}{2}\left( {\frac{1}{r}-\frac{1}{\tilde p} + \frac{1}{d}}\right)}+1\right) 2e^{-\beta(t-\tau)} d\tau\cr
 		&+&\tilde{ \lambda}(t)\int_{-1}^t  (t-\tau)^{-\frac{d}{2}\left( {\frac{1}{r}-\frac{1}{\tilde p} + \frac{1}{d}}\right)}2  |\tau|^{-\frac{d}{2}({\frac{1}{p}-\frac{1}{s^\prime}+\frac{1}{d}})}  d\tau\cr
 		&\le& M_{32a}  -2|t|^{\theta_r} \int_{-1}^t \left( \dfrac{\tau}{t}-1\right)  ^{-\frac{d}{2}\left( {\frac{1}{r}-\frac{1}{\tilde p} + \frac{1}{d}}\right)} \left(\dfrac{\tau}{t} \right) ^{-\frac{d}{2}({\frac{1}{p}-\frac{1}{s^\prime}+\frac{1}{d}})} d\left(\dfrac{\tau}{t} \right)  
 		,\;\, ( \hbox{because  }   \dfrac{\tau}{t} >1) 
 		\cr
 		&\le& M_{32a}  -2|t|^{\theta_r} \int_{-1}^t \left( \dfrac{\tau}{t}\right)  ^{-\frac{d}{2}\left( {\frac{1}{r}-\frac{1}{\tilde p} + \frac{1}{d}}\right)} \left(\dfrac{\tau}{t} \right) ^{-\frac{d}{2}({\frac{1}{p}-\frac{1}{s^\prime}+\frac{1}{d}})} d\left(\dfrac{\tau}{t} \right)  
 		\cr
 		&\le& M_{32a}  +2  .\dfrac{|t|^{\theta_r}-|t|^{1-\theta_{s^\prime}}}{1-\theta_r-\theta_{s^\prime}} \le  M_{32a}  +\dfrac{2}{1-\theta_r-\theta_{s^\prime}}:=M_{32b}.
 	\end{eqnarray*}

 	If $0<t\leq1$ then
 	\begin{eqnarray*} 
 		G_{32}(t) &\lesssim& \int_{-\infty}^{-1}   \left( (t-\tau)^{-\frac{d}{2}\left( {\frac{1}{r}-\frac{1}{\tilde p} + \frac{1}{d}}\right)}+1\right) 2e^{-\beta(t-\tau)} d\tau\cr
 		&+&\tilde{ \lambda}(t)\int_{-1}^0  (t-\tau)^{-\frac{d}{2}\left( {\frac{1}{r}-\frac{1}{\tilde p} + \frac{1}{d}}\right)}2  |\tau|^{-\frac{d}{2}({\frac{1}{p}-\frac{1}{s^\prime}+\frac{1}{d}})}  d\tau\cr
 		&+&\tilde{ \lambda}(t) \int_0^t   (t-\tau)^{-\frac{d}{2}\left( {\frac{1}{r}-\frac{1}{\tilde p} + \frac{1}{d}}\right)}\left( \tau^{-\frac{d}{2}({\frac{1}{p}-\frac{1}{s}+\frac{1}{d}})} + \tau^{-\frac{d}{2}({\frac{1}{p}-\frac{1}{\tilde p}+\frac{1}{d}})} \right) d\tau\cr
 		&\le& M_{32a}   +\tilde{ \lambda}(t)\int_{-1}^0  (-\tau)^{-\frac{d}{2}\left( {\frac{1}{r}-\frac{1}{\tilde p} + \frac{1}{d}}\right)}2  |\tau|^{-\frac{d}{2}({\frac{1}{p}-\frac{1}{s^\prime}+\frac{1}{d}})}  d\tau\cr
 		&+&\tilde{ \lambda}(t) \int_0^t   (t-\tau)^{-\frac{d}{2}\left( {\frac{1}{r}-\frac{1}{\tilde p} + \frac{1}{d}}\right)}2 \tau^{-\frac{d}{2}({\frac{1}{p}-\frac{1}{s}+\frac{1}{d}})}d\tau\cr
 		&\le& M_{32a}   +2t^{\frac{1}{2}-\frac{d}{2r}+ \frac{d}{2s^\prime}}  +2t^{\frac{1}{2}-\frac{d}{2p}} \mathbf B\left( \theta_r,\theta_s\right)  
 		\le  M_{32a}   +2 
 		+ 2\mathbf B\left( \theta_r,\theta_s\right) :=M_{32c}.
 	\end{eqnarray*}
 	
 	If $1< t$ then
 	\begin{eqnarray*} 
 		G_{32}(t) &\lesssim& \int_{-\infty}^{-1}    2e^{-\beta(t-\tau)} d\tau  + \int_{-1}^0   2  |\tau|^{-\frac{d}{2}({\frac{1}{p}-\frac{1}{s^\prime}+\frac{1}{d}})}  d\tau\cr
 		&+&\tilde{ \lambda}(t) \int_0^1   \left( (t-\tau)^{-\frac{d}{2}\left( {\frac{1}{r}-\frac{1}{\tilde p} + \frac{1}{d}}\right)}+1\right)  \left( \tau^{-\frac{d}{2}({\frac{1}{p}-\frac{1}{s}+\frac{1}{d}})} + \tau^{-\frac{d}{2}({\frac{1}{p}-\frac{1}{\tilde p}+\frac{1}{d}})} \right) d\tau\cr
 		&+& \int_{1}^{t}    2e^{-\beta(t-\tau)} d\tau \cr
 		&\le&  \dfrac{2}{\beta}+ \dfrac{2}{\theta_{s^\prime}}+    \int_0^1   \left( (1-\tau)^{-\frac{d}{2}\left( {\frac{1}{r}-\frac{1}{\tilde p} + \frac{1}{d}}\right)}+1\right)    2\tau^{-\frac{d}{2}({\frac{1}{p}-\frac{1}{s}+\frac{1}{d}})}  d\tau + \dfrac{2}{\beta}\cr 
 		&\le&  \dfrac{4}{\beta}+ \dfrac{2}{\theta_{s^\prime}}+ 2\mathbf B\left( \theta_r,\theta_{s}\right) + \dfrac{2}{\theta_{\tilde{p}} }:=M_{32d}.
 	\end{eqnarray*}

 	Therefore,  $ G_{32}(t)\lesssim M_{32} := \max\left\{M_{32b}, M_{32c},M_{32d}\right\}<+\infty$ for all $t\in \mathbb R.$
 	
 	{Moreover, we continue to give the estimation for the integrals  } 
 	\begin{align*}
 		G_{41}(t) &:= \hat{ \lambda}(t)\int_{-\infty}^t [c_d(t-\tau)]^{ \frac{1}{p}-\frac{1}{s} +\frac{1}{d}}e^{-\beta(t-\tau)} d\tau,\cr
 		G_{42}(t) &:=\hat{ \lambda}(t) \int_{-\infty}^t [c_d(t-\tau)]^{\frac{1}{s} + \frac{1}{d}}\left( \left[\hat{ \lambda}(\tau)\right] ^{-1} +\left[ \tilde{ \lambda}(\tau)\right] ^{-1} \right) e^{-\beta(t- \tau)} d\tau
 	\end{align*}
 as follows.	  For the boundedness of  $G_{41}$, 
 	
 	For the case $t\le -1$,
 	\begin{eqnarray*}
 		G_{41}(t) &\lesssim& \int_{-\infty}^t   (t-\tau)^{{-\frac{d}{2}}\left( \frac{1}{p}-\frac{1}{s}+\frac{1}{d} \right)} e^{-\beta(t-\tau)}d\tau =\dfrac{1}{\beta^{\theta_s}}\mathbf\Gamma \left(\theta_s\right):=M_{41a}.
 	\end{eqnarray*}
 	
 	For the case $-1<t\le0$,
 	\begin{eqnarray*}
 		G_{41}(t) &\lesssim& \int_{-\infty}^{-1}   (t-\tau)^{{-\frac{d}{2}}\left( \frac{1}{p}-\frac{1}{s}+\frac{1}{d} \right)} e^{-\beta(t-\tau)}d\tau 
 		+ \int_{-1}^t   (t-\tau)^{{-\frac{d}{2}}\left( \frac{1}{p}-\frac{1}{s}+\frac{1}{d} \right)} d\tau\cr 
 		&\le&\dfrac{1}{\beta^{\theta_s}}\mathbf\Gamma \left(\theta_s\right)+ \frac{1}{\theta_s}:=M_{41b}.
 	\end{eqnarray*}
 	
 	For the case $0<t\le1$,
 	\begin{eqnarray*}
 		G_{41}(t) &\lesssim& \int_{-\infty}^{-1} e^{-\beta(t-\tau)}d\tau 
 		+ \int_{-1}^t   (t-\tau)^{{-\frac{d}{2}}\left( \frac{1}{p}-\frac{1}{s}+\frac{1}{d} \right)} d\tau \cr 
 		&\le&\dfrac{1}{\beta } + \frac{(t+1)^{\theta_s}}{\theta_s}\le \dfrac{1}{\beta } + \frac{2^{\theta_s}}{\theta_s}:=M_{41c}.
 	\end{eqnarray*}
 	
 	For the case $t > 1$, 
 	\begin{eqnarray*}
 		G_{41}(t) &\lesssim& \int_{-\infty}^0 e^{-\beta(t-\tau)}d\tau 
 		+  \int_0^t \left[ (t-\tau)^{{-\frac{d}{2}}\left( \frac{1}{p}-\frac{1}{s}+\frac{1}{d} \right)} +1 \right] e^{-\beta(t-\tau)}d\tau\cr 
 		&\le&\dfrac{1}{\beta } +\dfrac{1}{\beta^{\theta_s}}\mathbf\Gamma \left(\theta_s\right)+ \frac{1}{\beta}:=M_{41d}.
 	\end{eqnarray*}
 	Hence,
 	$G_{41}(t)\lesssim M_{41}:= \max{\left\{M_{41b}, M_{41c}, M_{41d}\right\}}<+\infty$ for all $t\in \mathbb R.$

 	Finally, we give the estimation for $G_{42}$.

 	For the case $t\le -1$,
 	\begin{eqnarray*}
 		G_{42}(t) &\lesssim& \int_{-\infty}^t  \left(  (t-\tau)^{-\frac{d}{2}\left( {\frac{1}{s} + \frac{1}{d}}\right)}+1\right) 2 e^{-\beta(t-\tau)}d\tau  = \dfrac{2\mathbf\Gamma\left(\frac{1}{2}-\frac{d}{2s}\right)  }{\beta^{\frac{1}{2}-\frac{d}{2s}}}  +\dfrac{2}{\beta}:=M_{42a}.
 		\cr 
 	\end{eqnarray*}
 	
 	For the case $-1<t\le0$, 
 	\begin{eqnarray*}
 		G_{42}(t) &\lesssim& \int_{-\infty}^{-1}  \left(  (t-\tau)^{-\frac{d}{2}\left( {\frac{1}{s} + \frac{1}{d}}\right)}+1\right) 2 e^{-\beta(t-\tau)}d\tau 
 		\cr 
 		&+& \hat{ \lambda}(t)\int_{-1}^t   \left(  (t-\tau)^{-\frac{d}{2}\left( {\frac{1}{s} + \frac{1}{d}}\right)}+1\right) 2 |\tau|^{-\frac{d}{2}({\frac{1}{p}-\frac{1}{s^\prime}+\frac{1}{d}})} d\tau\cr 
 		&\leq& \dfrac{2\mathbf\Gamma\left(\frac{1}{2}-\frac{d}{2s}\right)  }{\beta^{\frac{1}{2}-\frac{d}{2s}}}  +\dfrac{2}{\beta}   +\hat{ \lambda}(t)\int_{-1}^t    \left(  |\tau|^{-\frac{d}{2}\left( {\frac{1}{s} + \frac{1}{d}}\right)}+1\right) 2 |\tau|^{-\frac{d}{2}({\frac{1}{p}-\frac{1}{s^\prime}+\frac{1}{d}})} d\tau
 		\cr 
 		&\leq& \dfrac{2\mathbf\Gamma\left(\frac{1}{2}-\frac{d}{2s}\right)  }{\beta^{\frac{1}{2}-\frac{d}{2s}}}  +\dfrac{2}{\beta}   + 2\left(   \dfrac{|t|^{\frac{1}{2}-\frac{d}{2s}}}{\frac{d}{2r}-\frac{d}{2s^\prime}}+\dfrac{1}{\theta_{s^\prime}}\right)\cr
 		&\le& \dfrac{2\mathbf\Gamma\left(\frac{1}{2}-\frac{d}{2s}\right)  }{\beta^{\frac{1}{2}-\frac{d}{2s}}}  +\dfrac{2}{\beta}   +      \dfrac{2}{\frac{d}{2r}-\frac{d}{2s^\prime}}+\dfrac{2}{\theta_{s^\prime}}:=M_{42b}.  
 	\end{eqnarray*}

 	For the case $0<t\le1$, 
 	\begin{eqnarray*}
 		G_{42}(t) &\lesssim& \int_{-\infty}^{-1}    2 e^{-\beta(t-\tau)}d\tau 
 		+ \hat{ \lambda}(t)\int_{-1}^0   \left(  (t-\tau)^{-\frac{d}{2}\left( {\frac{1}{s} + \frac{1}{d}}\right)}+1\right) 2 |\tau|^{-\frac{d}{2}({\frac{1}{p}-\frac{1}{s^\prime}+\frac{1}{d}})} d\tau \cr 
 		&+&\hat{ \lambda}(t)\int_0^t   \left(  (t-\tau)^{-\frac{d}{2}\left( {\frac{1}{s} + \frac{1}{d}}\right)}+1\right) \left(  \tau^{-\frac{d}{2}({\frac{1}{p}-\frac{1}{s}+\frac{1}{d}})} +\tau^{-\frac{d}{2}({\frac{1}{p}-\frac{1}{\tilde{p}}+\frac{1}{d}})}\right)  d\tau    \cr 
 		&\le& \int_{-\infty}^{-1}    2 e^{-\beta(t-\tau)}d\tau 
 		+ \hat{ \lambda}(t)\int_{-1}^0   \left(  |\tau|^{-\frac{d}{2}\left( {\frac{1}{s} + \frac{1}{d}}\right)}+1\right) 2 |\tau|^{-\frac{d}{2}({\frac{1}{p}-\frac{1}{s^\prime}+\frac{1}{d}})} d\tau \cr 
 		&+&\hat{ \lambda}(t)\int_0^t   \left(  (t-\tau)^{-\frac{d}{2}\left( {\frac{1}{s} + \frac{1}{d}}\right)}+1\right) 2 \tau^{-\frac{d}{2}({\frac{1}{p}-\frac{1}{s}+\frac{1}{d}})} d\tau    \cr
 		&\le&    \dfrac{2}{\beta}   +      \dfrac{2}{\frac{d}{2r}-\frac{d}{2s^\prime}}+\dfrac{2}{\theta_{s^\prime}}+t^{\frac{1}{2}-\frac{d}{2s}}\mathbf B\left(\frac{1}{2}-\frac{d}{2s},\theta_s\right) +\dfrac{2}{\theta_s} \cr
 		&\le&    \dfrac{2}{\beta}   +      \dfrac{2}{\frac{d}{2r}-\frac{d}{2s^\prime}}+\dfrac{2}{\theta_{s^\prime}}+\mathbf B\left(\frac{1}{2}-\frac{d}{2s},\theta_s\right) +\dfrac{2}{\theta_s}:=M_{42c}.  
 	\end{eqnarray*}

 	For the case $t>1$,
 	\begin{eqnarray*}
 		G_{42}(t)&\lesssim& \int_{-\infty}^{-1}    2 e^{-\beta(t-\tau)}d\tau 
 		+  \int_{-1}^0 2 |\tau|^{-\frac{d}{2}({\frac{1}{p}-\frac{1}{s^\prime}+\frac{1}{d}})} d\tau \cr 
 		&+&\int_0^1   \left(  (t-\tau)^{-\frac{d}{2}\left( {\frac{1}{s} + \frac{1}{d}}\right)}+1\right) 2 \tau^{-\frac{d}{2}({\frac{1}{p}-\frac{1}{s}+\frac{1}{d}})} d\tau   +\int_1^t    2 e^{-\beta(t-\tau)}d\tau \cr
 		&\le&    \dfrac{2}{\beta} +\dfrac{2}{\theta_{s^\prime}}+t^{\frac{1}{2}-\frac{d}{2p}}\mathbf B\left(\frac{1}{2}-\frac{d}{2s},\theta_s\right) +\dfrac{2}{\theta_s}+\dfrac{2}{\beta} \cr
 		&\le&    \dfrac{4}{\beta}    +\dfrac{2}{\theta_{s^\prime}}+\mathbf B\left(\frac{1}{2}-\frac{d}{2s},\theta_s\right) +\dfrac{2}{\theta_s}:=M_{42d}.  
 	\end{eqnarray*}
Therefore, $G_{42}(t) \lesssim M_{42}:=\max\left\{ M_{42b}, M_{42c}, M_{42d}\right\}<+\infty$ for all $t\in \mathbb R.$\\
{\bf Conflict of interest.} The authors declare that there are no conflicts of interest.

\end{document}